\documentclass[12pt]{amsart}

\usepackage[a4paper,text={16cm,24cm},centering]{geometry}
\usepackage{latexsym}
\usepackage{graphicx}
\usepackage{amsmath}
\usepackage{amssymb}
\usepackage{mathrsfs}
\usepackage{bm}

\newtheorem{proposition}{Proposition}[section]
\newtheorem{theorem}[proposition]{Theorem}
\newtheorem{corollary}[proposition]{Corollary}
\newtheorem{lemma}[proposition]{Lemma}

\theoremstyle{definition}
\newtheorem{definition}[proposition]{Definition}

\theoremstyle{remark}

\newtheorem{remark}[proposition]{Remark}
\newtheorem{example}[proposition]{Example}
\numberwithin{equation}{section}

\def\ind{\operatorname{ind}}
\def\R{\mathbb R}
\def\({\left(}
\def\){\right)}
\def\Dt{\partial_t}
\def\Dx{\Delta_x}
\def\Nx{\nabla_x}
\def\tto{\big|_{t=0}}
\def\DOM{\big|_{\partial\omega}}
\def\<{\left<}
\def\>{\right>}
\def\hat{\widehat}
\def\Ree{\operatorname{Re}}
\def\Imm{\operatorname{Im}}
\def\eb{\varepsilon}
\def\dist{\operatorname{dist}}
\def\Cal{\mathcal}
\def\Bbb{\mathbb}
\begin{document}
\title[Elliptic attractors and the parabolic limit]{
Attractors for the nonlinear elliptic boundary value problems and
their parabolic singular limit}

\author [M. Vishik and S. Zelik]{Mark I. Vishik${}^1$ and Sergey V. Zelik${}^{1,2}$}
\date{\today}
\address{
${}^1$ Institute for Information Transmission Problems (Kharkevich
Institut), Russian Academy of Sciences, Bolshoi Karetniy~19, Moscow
127\,994, GSP-4, Russia}

\address{
${}^2$  University of Surrey, Department of Mathematics, \newline
Guildford, GU2 7XH, United Kingdom }
\begin{abstract}
We apply the dynamical approach to the study of the second order
semi-linear elliptic boundary value problem
 in a cylindrical domain
with a small parameter $\eb$ at the second derivative with respect to
the variable $t$ corresponding  to the axis of the cylinder.
We prove that, under natural assumptions on the nonlinear interaction
function $f$ and the external forces $g(t)$, this problem possesses
the uniform attractor $\Cal A_\eb$ and that these attractors tend
as $\eb\to0$ to the attractor $\Cal A_0$ of the limit parabolic
equation. Moreover, in case where the limit attractor $\Cal A_0$ is
regular, we give the detailed description of the structure of
the uniform attractor $\Cal A_\eb$, if $\eb>0$ is small enough, and
estimate the symmetric distance between the attractors $\Cal A_\eb$
and~$\Cal A_0$.
\end{abstract}


\maketitle
\tableofcontents

\section{Introduction}\label{s0}
We consider the following semi-linear elliptic boundary value
problem in an infinite cylinder $\Omega:=\R\times\omega$:
\begin{equation}\label{0.1}
a(\Dt^2 u+\Dx u)-\eb^{-1}\gamma\Dt u-f(u)=g(t),\ \ (t,x)\in\Omega,\ \
u\DOM=0,
\end{equation}
where $\omega\subset\subset\R^n$ is a bounded domain of $\R^n$,
$u=(u_1,\cdots,u_k)$ is an unknown vector-valued function, $a$ and
$\gamma$ are given constant matrices which satisfy $a+a^*>0$ and
 $\gamma=\gamma^*>0$, $f$ and $g$ are given nonlinear interaction
function and the external forces respectively which satisfy some
natural assumptions (formulated in Section \ref{s1}) and $\eb>0$ is a small
parameter.
\par
Elliptic boundary problems of the form \eqref{0.1} appear, e.g.
under studying the equilibria or the traveling waves for the
corresponding evolution equations of mathematical physics. For
instance, let us consider the following reaction-diffusion system
in the unbounded cylindrical domain $\Omega$:
\begin{equation}\label{0.2}
\partial_\eta v=a\Delta_{(t,x)}v-f(v)-g(t-\eb^{-1}\gamma\eta),\
\ v\big|_{\partial\Omega}=0,
\end{equation}
where the variable $t\in\R$ remains to be
 spatial, the variable $\eta$ plays the
role of physical time and $\gamma$ is a diagonal matrix. Thus,
the external forces $g(\eta,t,x):=g(t-\eb^{-1}\gamma\eta,x)$
 in \eqref{0.2} have the form of a fast
traveling (along the axis of the cylinder) wave
 (with the wave speed $\eb^{-1}\gamma\gg1$). Then, the problem
of finding the traveling wave solution
$v(\eta,t,x):=u(t-\eb^{-1}\gamma\eta,x)$ of equation \eqref{0.2}
which is modulated by the traveling wave external forcing,
 obviously, reduces to the study of elliptic problem \eqref{0.1}.
Another natural example is the following reaction-diffusion system
in the cylinder $\Omega$:
\begin{equation}\label{0.3}
\partial_\eta v=a\Delta_{(t,x)}v-\eb^{-1}\gamma\Dt v-f(v)-g(t),\ \
 (t,x)\in\Omega, \ \
u\big|_{\partial\Omega}=0
\end{equation}
with the strong drift along the axis of the cylinder (which is
described by the transport term $\eb^{-1}\gamma\Dt v$). Then,
\eqref{0.1} is the equation on equilibria for problem~\eqref{0.3}.
\par
It is convenient to scale from the very beginning the variable $t$ as
follows: $t':=\eb^{-1}t$. Then, problem \eqref{0.1} reads
\begin{equation}\label{0.4}
a(\eb^2\Dt u+\Dx u)-\gamma\Dt u-f(u)=g_\eb(t),\ \ u\DOM=0,\ \
g_\eb(t):=g(\eb^{-1}t),
\end{equation}
where we denote the new variable $t'$ by $t$ again for simplicity.
\par
We are interested in the global structure of the set of bounded
(with respect to $t\to\pm\infty$) solutions of problem
\eqref{0.4}. To this end, we use the so-called dynamical approach
for the study of elliptic boundary value problems in cylindrical
domains which has been initiated in \cite{6} and \cite{16},
 see also \cite{2,3}, \cite{7}, \cite{10}, \cite{11},
\cite{19,191,20,21,22,23}, \cite{26,27,28}
 and the references therein for its further development.
 Following this approach, we
introduce, for every $\tau\in\R$, the auxiliary elliptic boundary
value problem:
\begin{equation}\label{0.5}
\begin{cases}
a(\eb^2\Dt u+\Dx u)-\gamma\Dt u=g_\eb(t),\ \ (t,x)\in\Omega_+^\tau,\\
u\DOM=0,\ \ u\big|_{t=\tau}=u_\tau,
\end{cases}
\end{equation}
in the half-cylinder $\Omega_+^\tau:=(\tau,+\infty)\times\omega$
equipped by the additional boundary condition $u|_{t=\tau}=u_\tau$ at
the origin of the half-cylinder $\Omega_+^\tau$ and the function
$u_\tau$ is assumed to belong to the appropriate functional space
 $V^p_\eb(\omega)$ which will be specified in Section \ref{s1}. If
problem \eqref{0.5} possesses a unique (bounded as $t\to+\infty$)
solution (in certain functional class), for every $u_\tau\in
V^p_\eb(\omega)$, then \eqref{0.5} defines a dynamical process
 $\{U^\eb_{g_\eb}(t,\tau),\, t,\tau\in\R,\, t\ge\tau\}$ via
\begin{equation}\label{0.6}
U^\eb_{g_\eb}(t,\tau)u_\tau:=u(t),\ \ \text{where $u(t)$ solves
\eqref{0.5}},\ \ U^\eb_{g_\eb}(t,\tau): V^p_\eb(\omega)\to V^p_\eb(\omega).
\end{equation}
Moreover, if this dynamical process possesses a global (uniform)
attractor $\Cal A_\eb$, then this attractor is generated by all
bounded (with respect to $t\to\pm\infty$) solutions of the initial
problem \eqref{0.4} (and all its shifts along the $t$ axis, together
with their closure in the corresponding topology, see Section \ref{s3} for
the details). Thus, studying  of the bounded solutions of \eqref{0.4}
is, in a sense, equivalent to the study of the attractor $\Cal A_\eb$ of
auxiliary dynamical process \eqref{0.6}.
\par
In the present paper, we give a detailed study of
auxiliary problems \eqref{0.5} in case $\eb$ is small enough
and investigate their behavior as $\eb\to0$.
The paper is organized as follows.
 The existence of
a bounded solution $u(t)$ of problem \eqref{0.5} and several important
estimates are derived in Section \ref{s1}. The uniqueness of this solution
is verified in Section \ref{s2} under the assumption that $\eb$ is small
enough. Moreover, we show there that the dynamical process
\eqref{0.6} associated with problem \eqref{0.5} is
uniformly (with respect to $\eb$)
Frechet
differentiable with respect to the 'initial data' $u_\tau\in
V^p_\eb(\omega)$. The existence of the uniform attractor $\Cal A_\eb$
for the process \eqref{0.6} is established in Section \ref{s3}.  Moreover,
we prove there that, for rather wide class of the external
forces $g$, the attractors $\Cal A_\eb$ converge as $\eb\to0$ (in the
sense of upper semicontinuity) to the attractor $\Cal A_0$ of the
limit parabolic problem
\begin{equation}\label{0.7}
\gamma\Dt u-a\Dx u+f(u)=g_0(t),\ \ u\DOM=0,\ \ u\big|_{t=\tau}=u_\tau,
\end{equation}
where the limit external forces $g_0(t)$ average the external
 forces $g_\eb(t):=g(\eb^{-1}t)$ of problems \eqref{0.5}. In
 particular, the class of admissible external forces $g$ contains
the autonomous external forces: $g(t)\equiv g_0$,
heteroclinic profiles:
\begin{equation}\label{0.8}
g(t)\to g_\pm\ \ \text{ as $t\to\pm\infty$ and $g_\pm$ are independent
of $t$},
\end{equation}
solitary waves ($g_+=g_-$ in \eqref{0.8}), periodic, quasiperiodic
and almost-periodic with respect to $t$ external forces $g$
 and even some classes of  non
almost-periodic oscillations, see Examples \ref{Ex3.1}--\ref{Ex3.3}.
\par
Furthermore, in Section \ref{s4}, we prove that dynamical processes
\eqref{0.6} tend as $\eb\to0$ to the process $U^0_{g_0}(t,\tau)$
associated with limit parabolic problem \eqref{0.7} and obtain
the quantitative bounds for that convergence in terms of the parameter
$\eb$.
\par
In Sections \ref{s5} and \ref{s6}, we restrict ourselves to consider only the case of
almost-periodic external forces $g(t)$ in the right-hand side of
equation \eqref{0.5}. In this case, limit parabolic equation
\eqref{0.7}  is autonomous
\begin{equation}\label{0.9}
g_0(t):=\bar g,
\end{equation}
where $\bar g$ is the mean of almost-periodic function $\bar g$. We
also assume that the global attractor $\Cal A_0$ of the limit
parabolic equation is regular (it will be so if this equation
possesses a global Lyapunov function and all of the equilibria are
hyperbolic, see Section \ref{s6} for the details). Then, using the theory
of non-autonomous perturbations of regular attractors developed in
\cite{12,13,29}, we establish the existence of the non-autonomous
regular attractor for problems \eqref{0.6} if $\eb$ is small
enough. In this case, the attractors $\Cal A_\eb$ are occurred
not only upper semicontinuous, but also lower semicontinuous as
$\eb\to0$ and we give the quantitative bounds for the symmetric
distance between them in terms of the perturbation parameter $\eb$.
In particular, we prove there that equation \eqref{0.4} possesses
the finite number of different almost-periodic (with respect to
$t$) solutions and that every other bounded solution of that
equation is a heteroclinic orbit between two different
almost-periodic solutions. We also recall that the regular
attractor for system \eqref{0.5} with $\eb=1$, $\gamma\gg1$ and
{\it autonomous} external forces $g_\eb$ has been considered in our
previous paper~\cite{27}. Moreover, the estimates for the {\it
nonsymmetric} Hausdorff distance between the attractors $\Cal
A_\eb$ and $\Cal A_0$ in terms of the parameter $\eb$ have been
obtained in \cite{28} for the case where the attractor $\Cal A_0$
of the limit parabolic equation is regular and the external forces
$g_\eb(t)=g(\eb^{-1}t)$ are almost-periodic with respect to $t$.
\par
Finally, several uniform (with respect to $\eb$) estimates
for the linear equation of the form \eqref{0.4}
which are systematically used throughout of the paper
are gathered in Appendix.
\par
{\bf Acknowledgements:}

This research was partially supported by the Russian Foundation of
Basic Researches (projects 11-01-00339 and 10-01-00293).


\section{Uniform (with respect to $\eb\to0$) a priori estimates}\label{s1}
In this section, we consider the following nonlinear elliptic
boundary value problem in a half cylinder
$\Omega_+^\tau:=[\tau,+\infty)\times\omega$, $\tau\in\R$:
\begin{equation}\label{1.1}
\begin{cases}
a(\eb^2\Dt^2u+\Dx u)-\gamma\Dt u-f(u)=g(t),\ \ (t,x)\in\Omega_+^\tau,\\
u\DOM=0,\ \ \ u\big|_{t=\tau}=u_\tau,
\end{cases}
\end{equation}
where $\omega\subset\subset\R^n$ is a bounded domain of $\R^n$ with a
sufficiently smooth boundary,
$u=(u^1,\cdots,u^k)$ is an unknown vector-valued function,
$\Dx$ is the Laplacian with respect to $x$,
$a$ and $\gamma$ are given constant $k\times k$-matrices satisfying
$a+a^*>0$ and $\gamma=\gamma^*>0$, $f(u)$ is a given nonlinear
function which satisfies the following assumptions:
\begin{equation}\label{1.2}
\begin{cases}
1.\ \ f\in C^2(\R^k,\R^k),\\
2.\ \ f(v).v\ge -C,\ \ f'(v)\ge -K,\ \ \forall v\in\R^k,\\
3.\ \ |f(v)|\le C(1+|v|^q),\ \forall v\in\R^k,\  \
q<q_{max},
\end{cases}
\end{equation}
where  $v.w$ stands for the inner product of the
vectors $v\in\R^k$ and $w\in\R^k$ and the critical growth exponent $q_{max}$ equals infinity for $n=1$ or $n=2$ and $q_{max}:=\frac{n+2}{n-2}$ for  $n\ge3$. In order to formulate our assumptions on the solution
$u(t)$, the external forces $g(t)$ and the initial data $u_0$, we
need to define the appropriate functional spaces.
\begin{definition}\label{Def1.1} For every $l\in\R_+$ and $s\in[2,\infty)$,
we define the following spaces:
\begin{equation}\label{1.3}
W^{l,s}_b(\Omega_+^\tau):=\{u\in D'(\Omega_+^\tau),\ \
\|u\|_{W^{l,s}_b}:=\sup_{T\ge\tau}\|u\|_{W^{l,s}(\Omega_T)}<\infty\},
\end{equation}
where $\Omega_T:=(T,T+1)\times\omega$ and $W^{l,s}$ denotes the
ordinary Sobolev space of functions whose derivatives up to order
$l$ belong to $L^s$, see \cite{25}. In particular, we write in the
sequel $L^s_b(\Omega^\tau_+)$ instead of
$W^{0,s}_b(\Omega_+^\tau)$.
\par
Moreover, we also introduce the following spaces associated with
the linear part of equation \eqref{1.1}:
\begin{multline}\label{1.4}
W^{(1,2),s}_\eb(\Omega_T):=\{u\in D'(\Omega_T),\ \
\|u\|_{W^{(1,2),s}_\eb}:=\\=
\eb^2\|\Dt^2u\|_{L^s(\Omega_T)}+\|\Dt u\|_{L^s(\Omega_T)}+
\|u\|_{L^s([T,T+1],W^{2,s}(\omega))}<\infty,\ \ u\DOM=0\}
\end{multline}
and, analogously to \eqref{1.3}
\begin{equation}\label{1.45}
W^{(1,2),s}_{\eb,b}(\Omega_+^\tau):=\{u\in D'(\Omega_+^\tau),\
\|u\|_{W^{(1,2),s}_{\eb,b}}:=\sup_{T\ge\tau}
\|u\|_{W^{(1,2),s}_\eb(\Omega_T)}<\infty,\ \ u\DOM=0\}.
\end{equation}
We also introduce the uniform with respect to $\eb$ trace space
(at $t=\tau$)
of functions belonging to the space \eqref{1.45}:
\begin{multline}\label{1.5}
V_\eb^s(\omega):=\{u\in D'(\omega),\\
\|u\|_{V^p_\eb}:=\|u\|_{W^{2(1-1/s),s}(\omega)}+
\eb^{1/s}\|u\|_{W^{2-1/s,s}(\omega)}<\infty,\
u\DOM=0\}.
\end{multline}
We note that, for $\eb>0$, the space
$W^{(1,2),s}_{\eb,b}(\Omega_+^\tau)$ is equivalent to
 $W^{2,s}_b(\Omega_+^\tau)$ and, for $\eb=0$ this space
coincides with the anisotropic Sobolev-Slobodetskij space
 $W^{(1,2),s}_b(\Omega_+^\tau)$ which corresponds to a second order
parabolic operator.  Analogously,  for $\eb>0$, the space $V_\eb^s(\omega)$ coincides with the trace space at $t=\tau$ of
of the classical Sobolev space $W^{2,s}(\Omega_\tau)$ and, for $\eb=0$, we have the trace space at $t=\tau$, for the anisotropic space $W^{(1,2),s}(\Omega_\tau)$,
 see e.g. \cite{5}. The dependence of the norms \eqref{1.4} and \eqref{1.5} on the parameter $\eb$ is
  chosen in such way that the constants in the proper maximal regularity estimates and the trace theorems will be independent of $\eb\to0$, see Appendix below.
\end{definition}
\begin{definition}\label{Def1.sol} We restrict ourselves to consider only such solutions $u(t)$ of
problem \eqref{1.1} which remain bounded as $t\to+\infty$. To be more precise, a function $u$ will be a solution of problem \eqref{1.1} if $u$
 belongs to the space
 $W^{(1,2),p}_{\eb,b}(\Omega_+^\tau)$ for some  $p>p_{min}:=\max\{2,(n+2)/2\}$, and satisfies the equation and the boundary conditions in the sense of distributions. Then, we should require the initial data $u_0$ belongs to the trace space $V^p_\eb(\omega)$
 and the external forces $g\in L^p_b(\Omega_+^\tau)$.
 \par
Note that the assumption $p>p_{max}$ guarantees that $u\in C_b(\Omega_+^\tau)$ for all $\eb$ (including the limit case $\eb=0$) and, therefore, the non-linearity $f(u)$ is well-defined and belongs to $C_b(\Omega_+^\tau)$ for all $u\in W^{(1,2),p}_{\eb,b}(\Omega_+^\tau)$.
\end{definition}

The main
result of this section is the following theorem.
\begin{theorem}\label{Th1.1} Let the above assumptions hold. Then, for every
$\eb\in[0,1]$ and
$u_\tau\in V^p_\eb(\omega)$, problem \eqref{1.1} has at least one
solution $u\in W^{(1,2),p}_{\eb,b}(\Omega_+^\tau)$ and the following
estimate hold, for every such solution:
\begin{equation}\label{1.6}
\|u\|_{W^{(1,2),p}_{\eb,b}(\Omega_T)}\le
Q(\|u_\tau\|_{V^p_\eb(\omega)})e^{-\alpha(T-\tau)}+
Q(\|g\|_{L^p_b(\Omega_+^\tau)}),
\end{equation}
where the constant $\alpha>0$ and the monotonic function
$Q:\R_+\to\R_+$ are independent of $\eb\in[0,1]$, $u_0\in
V^p_\eb(\omega)$, $\tau\in\R$, $T\ge\tau$ and $g\in L^p_b(\Omega_+^\tau)$.
\end{theorem}
\begin{proof} We first prove the analogue of estimate \eqref{1.6} for
$p=2$.
\begin{lemma}\label{Lem1.1} Let $u(t)$ be a solution of \eqref{1.1} (in the sense of Definition \ref{Def1.sol}). Then,
the following estimate holds:
\begin{equation}\label{1.7}
\|u\|_{W^{(1,2),2}_\eb(\Omega_T)}\le
Q(\|u_\tau\|_{V_\eb^p(\omega)})
e^{-\alpha(T-\tau)}+Q(\|g\|_{L^2_b(\Omega_+^\tau)}),
\end{equation}
where the constant $\alpha>0$ and the monotonic function $Q$
 are independent of
$\tau\in\R$, $T\ge\tau$, $\eb\in[0,1]$, $u_0$ and  $g$.
\end{lemma}
\begin{proof} We set $\phi_T(t):=e^{-\alpha|t-T|}$, where $T\in\R$ and
$\alpha>0$ is a small parameter which will be specified below,
multiply equation \eqref{1.1} by $\phi_T(t)u(t)$ and integrate
over $\Omega_+^\tau$. Then, after integrating by parts and using that $\gamma=\gamma^*$, we have
\begin{multline}\label{1.add1}
\<\eb^2a\Dt u.\Dt u+a\Nx u.\Nx u,\phi_T\>_\tau+\<f(u).u,\phi_T\>_\tau=1/2\<\gamma u.u,\phi'_T\>_\tau-\\-\eb^2\<a\Dt u.u,\phi'_T\>_\tau-\<g.u,\phi_T\>_\tau+\eb^2(a\Dt u(\tau).u(\tau),\phi_T(\tau))-1/2(\gamma u(\tau).u(\tau),\phi_T(\tau)),
\end{multline}
where $(u,v):=\int_{\omega}u(x).v(x)\, dx$ and $\<v,w\>_\tau:=\int_\tau^\infty\int_{\omega}v(t,x).w(t,x)\,dx\,dt$.
\par
Using now that
$a+a^*>0$,  $f(v).v\ge-C$ and
the obvious
inequality
\begin{equation}\label{1.8}
|\Dt \phi_T(t)|\le \alpha\phi_T(t),\ \ t\in\R,
\end{equation}
together with the  Cauchy-Schwartz inequality,
we have
\begin{multline}\label{1.9}
\<\eb^2|\Dt u|^2+|\Nx u|^2,\phi_T\>_\tau\le C\eb^2\alpha\<|\Dt
u|\cdot|u|,\phi_T\>_\tau+C\alpha\<|u|^2,\phi_T\>_\tau
\\+C\(1+\<|g|^2,\phi_T\>_\tau+
\phi_T(\tau)\|u_\tau\|_{L^2(\omega)}^2+\eb^2\phi_T(\tau)
\|u_\tau\|_{L^2(\omega)}
\|\Dt u(\tau)\|_{L^2(\omega)}\),
\end{multline}
where
 the constant $C$ is independent of $\eb$, $\tau$, $T$ and $\alpha$
(we also note that all of the integrals in \eqref{1.9} have a sense
since the solution $u$ is assumed to belong to $W^{(1,2),p}_{\eb,b}(\Omega_+^\tau)$).
Estimate \eqref{1.9} implies that, for sufficiently small
(but independent of $\eb$),
 $\alpha>0$
\begin{multline}\label{1.10}
\<\eb^2|\Dt u|^2+|\Nx u|^2,\phi_T\>_\tau\le \\
C_1\(1+\|g\|_{L^2_b(\Omega^\tau_+)}^2+
\phi_T(\tau)\|u_\tau\|_{L^2(\omega)}^2+
\eb^2\phi_T(\tau)\|u_\tau\|_{L^2(\omega)}
\|\Dt u(\tau)\|_{L^2(\omega)}\),
\end{multline}
where the constant $C_1$ is independent of $\eb$, $\tau$ and $T$. We
now set
\begin{equation*}
\Cal L_\eb u:=\eb^2\Dt^2u+\Dx u,\ \bar{\Cal L}_\eb
u:=\eb^2\Dt(\phi_T\Dt u)+\phi_T\Dx u\equiv \phi_T\Cal L_\eb u+\eb^2\phi_T'(t)\Dt u,
\end{equation*}
multiply equation \eqref{1.1} by $\bar{\Cal L}_\eb u$ and integrate
over $\Omega_+^\tau$. Then, integrating by parts, using that
 $\gamma=\gamma^*$, we have
\begin{multline}
\<a\Cal L_\eb u.\Cal L_\eb u,\phi_T\>_\tau-\eb^2\<a\Cal L_\eb u.\Dt u,\phi'_T\>_\tau-\frac{1}2\<\eb^2\gamma\Dt u.\Dt u+\gamma\Nx u.\Nx u,\phi'_T\>_\tau+\\+\frac12\phi_T(\tau)[\eb^2(\gamma\Dt u(\tau),\Dt u(\tau))-(\gamma\Nx u(\tau),\Nx u(\tau))]+\<f'(u)\Dt u.\Dt u,\phi_T\>_\tau+\\+
\<f'(u)\Nx u.\Nx u,\phi_T\>_\tau-\eb^2\phi_T(\tau)(f(u(\tau)),\Dt u(\tau))=\<h.\Cal L_\eb u,\phi_T\>_\tau+\eb^2\<h.\Dt u,\phi_T'\>_\tau
\end{multline}
and using that $a+a^*>0$, $\gamma>0$ and $f'(v)\ge -K$, together with estimate \eqref{1.8},
 after the straightforward estimates, we end up with
\begin{multline}\label{1.11}
\<|\Cal L_\eb u|^2,\phi_T\>_\tau+\eb^2\phi_T(\tau)\|\Dt
u(\tau)\|_{L^2(\omega)}^2
\le C_2\big(\<\eb^2|\Dt u|^2+|\Nx u|^2,\phi_T\>_\tau+\\+
\<|g|^2,\phi_T\>_\tau+\phi_T(\tau)\|\Nx u_\tau\|_{L^2(\omega)}^2+
\eb^2\phi_T(\tau)\|f(u_\tau)\|_{L^2(\omega)}^2\big),
\end{multline}
where the constant $C_2$ is independent of $\eb$, $\tau$ and $T$.
Applying estimate \eqref{1.10} in order to estimate the first term
in the right-hand side of \eqref{1.11}, we obtain, after the obvious
estimates that
\begin{multline}\label{1.12}
\<|\Cal L_\eb u|^2,\phi_T\>_\tau+\eb^2\phi_T(\tau)\|\Dt
u(\tau)\|_{L^2(\omega)}^2+\<\eb^2|\Dt u|^2+|\Nx u|^2,\phi_T\>_\tau\le
\\\le C_3\(1+\|g\|_{L^2_b(\Omega_+^\tau)}^2+
\phi_T(\tau)\|u_\tau\|_{W^{1,2}(\omega)}^2+
\eb^2\phi_T(\tau)\|f(u_\tau)\|_{L^2(\omega)}^2\),
\end{multline}
where the constant $C_3$ is independent of $\eb$, $\tau$ and $T$.
We now claim that
\begin{multline}\label{1.13}
\eb^4\|\Dt^2 u\|_{L^2(\Omega_T)}^2+\|\Dx u\|_{L^2(\Omega_T)}^2\le\\\le
C_4\(\<|\Cal L_\eb u|^2,\phi_T\>_\tau+\eb^2\<|\Dt u|^2+|u|^2,\phi_T\>_\tau+
\eb\phi_T(\tau)\|u_\tau\|_{W^{3/2,2}(\omega)}^2\),
\end{multline}
where $C_4$ is independent of $\eb$, $\tau$ and $T\ge\tau$. Indeed,
let $\varphi(t)\in C^\infty_0(R)$ be a cut-off function such that
$\varphi(t)=1$, for $t\in[0,1]$, and $\varphi(t)=0$, for $t\notin[-1,2]$.
For every $T\ge\tau$, we set $\varphi_T(t):=\varphi(t-T)$ and
$u_T(t):=\varphi_T(t)u(t)$. Then, the last function satisfies the
following equation:
\begin{equation*}
\Cal L_\eb u_T(t)=h_u(t):=\varphi_T(t)\Cal L_\eb u(t)+
2\eb^2\varphi_T'(t)\Dt u(t)+\eb^2\varphi_T''(t)u(t)
\end{equation*}
Applying Lemma \ref{LemA.1} with $p=2$ (see Appendix) to this equation,
we have
\begin{multline*}
\eb^4\|\Dt^2u_T\|_{L^2(\Omega_+^\tau)}^2+\|\Dx u_T\|_{L^2(\Omega_+^\tau)}^2\le
C\(\|h_u\|_{L^2(\Omega_+^\tau)}^2+
\eb\|u_T(\tau)\|_{W^{3/2,2}(\omega)}^2\)\\
\le C'\(\<|\Cal L_\eb u|^2,\phi_T\>_\tau+\eb^2\<|\Dt u|^2+|u|^2,\phi_T\>_\tau+
\eb\phi_T(\tau)\|u_\tau\|_{W^{3/2,2}(\omega)}^2\)
\end{multline*}
and estimate \eqref{1.13} is an immediate corollary of this estimate.
\par
Inserting now estimate \eqref{1.12} into the right-hand side of
\eqref{1.13}
and using the embedding $W^{2(1-1/p),p}(\omega)\subset C(\omega)$ (due to our
choice of the exponent $p$), we have
\begin{equation}\label{1.14}
\eb^4\|\Dt^2 u\|_{L^2(\Omega_T)}^2+\|\Dx u\|_{L^2(\Omega_T)}^2\le
C_5(1+\|g\|_{L^2_b(\Omega_+^\tau)}^2)+
Q(\|u_\tau\|_{V^p_\eb(\omega)})e^{-\alpha(T-\tau)},
\end{equation}
where the constant $C_5$ and the monotonic function $Q$ are
independent of $\eb$, $\tau$ and $T\ge\tau$.
\par
Thus, there only remains to estimate the $L^2$-norm of $\Dt u$. In
order to do so, we rewrite elliptic system \eqref{1.1} in the
following form:
\begin{equation}\label{1.15}
\gamma\Dt u=a\Dx u-f(u)+\tilde h_u(t),\ \ u\DOM=0,\ \ \tilde h_u(t):=\eb^2a\Dt^2u(t)-g(t).
\end{equation}
Equation \eqref{1.15} has the form of a nonlinear reaction-diffusion
system in the bounded domain $\omega$ with the non-autonomous external forces
$\tilde h_u(t)$ belonging to $L^2_b((\tau,\infty)\times\omega)$ (due to
estimate \eqref{1.14}).
Moreover, the nonlinearity $f(u)$ satisfies the quasimonotonicity
assumption $f'(v)\ge-K$.
 Consequently, multiplying \eqref{1.15} by
$\Dx u(t)$, integrating over $x$ and applying the Gronwall's
inequality, we derive  (in a standard way, see e.g. \cite{9}) that
\begin{equation}\label{1.16}
\|u(T)\|_{W^{1,2}(\omega)}^2\le C\|u_\tau\|_{W^{1,2}(\omega)}^2
e^{-\alpha(T-\tau)}+C+\int_\tau^T
e^{-\alpha(T-t)}\|\tilde h_u(t)\|_{L^2(\omega)}^2\,dt,
\end{equation}
where the positive constants $\alpha$ and $C$ are independent of
$\tilde h_u$. Using estimate \eqref{1.14} for estimating the last term in
the right-hand side of \eqref{1.16}, we have
\begin{multline}\label{1.17}
\eb^4\|\Dt^2 u\|_{L^2(\Omega_T)}^2+\|u\|_{L^2((T,T+1),W^{2,2}(\omega))}^2+
\|u\|_{L^\infty((T,T+1),W^{1,2}(\omega))}^2
\le\\\le
C_6(1+\|g\|_{L^2_b(\Omega_+^\tau)}^2)+
Q(\|u_\tau\|_{V^p_\eb(\omega)})e^{-\alpha(T-\tau)},
\end{multline}
where the constant $C_6$ and the monotonic function $Q$ are
independent of $\eb$, $\tau$ and~$T$.
\par
We now recall that, according to the embedding theorem, see e.g.
\cite{17} and \cite{25}
\begin{equation}\label{1.18}
\|u\|_{L^{2q_{max}}(\Omega_T)}\le
C\(\|u\|_{L^\infty((T,T+1),W^{1,2}(\omega))}+
\|u\|_{L^2((T,T+1),W^{2,2}(\omega))}\),
\end{equation}
where the exponent $q_{max}$ is the same as in \eqref{1.2}. Estimates
\eqref{1.17} and \eqref{1.18}, together with the growth restriction
\eqref{1.2}, imply that
\begin{equation}\label{1.19}
\|f(u)\|_{L^2(\Omega_T)}\le
Q_1(\|u_\tau\|_{V^p_\eb(\omega)})e^{-\alpha(T-\tau)}+
Q_1(\|g\|_{L^2_b(\Omega_+^\tau)}),
\end{equation}
where the constant $\alpha>0$ and the monotonic function $Q_1$ are
independent of $\eb$, $\tau$ and $T\ge\tau$. Expressing now $\Dt u$
from equation \eqref{1.1} and using estimates \eqref{1.17} and
\eqref{1.19}, we obtain the desired estimate for $\Dt u$ and finish
the proof of Lemma \ref{Lem1.1}.
\end{proof}
We are now ready to prove estimate \eqref{1.6}, for $p>2$. We consider only the more complicated case $n\ge3$ and rest the simpler case $n\le 2$ to the reader.
\par
 Remind that the nonlinearity $f(u)$ satisfies  growth
restriction given by the third formula of \eqref{1.2} where the exponent $q$ is {\it strictly}
less than $q_{max}$ and, due to the embedding theorem for the anisotropic spaces
$$
W_{\eb}^{(1,2),2}(\Omega_T)\subset W^{(1,2),2}_{0}(\Omega_T)\subset L^{q_{max}}_b(\Omega_T),
$$
see \cite{25},
and,
consequently, due to \eqref{1.7},
estimate \eqref{1.19} can be improved as
follows:
\begin{multline}\label{1.1$9'$}
\|f(u)\|_{L^{2+\delta_0}(\Omega_T)}\le C(1+\|u\|_{L^{q_{max}}(\Omega_T)})^q\le\\\le C_1(1+\|u\|_{W_{\eb}^{(1,2),2}(\Omega_T)})^q\le
Q(\|u_\tau\|_{V^p_\eb(\omega)})e^{-\alpha(T-\tau)}+
Q(\|g\|_{L^2_b(\Omega_+^\tau)}),
\end{multline}
where
$\delta_0:=\frac{2(q_{max}-q)}q>0$ and the constant $\alpha>0$ and the
monotone function $Q$ are independent of $\eb$, $\tau$ and $T$. We now rewrite
equation \eqref{1.1} in the following way:
\begin{equation}\label{1.20}
a(\eb^2\Dt^2 u+\Dx u)-\gamma\Dt u=H_u(t):=g(t)+f(u(t)),\ u\DOM=0,\
u\big|_{t=\tau}=u_\tau
\end{equation}
and apply the maximal elliptic regularity estimate of
Corollary \ref{CorA.1} to this linear equation which reads
$$
\|u\|_{W^{(1,2),2+\delta_0}_\eb(\Omega_T)}^{2+\delta_0}\le C\|u_\tau\|_{V_\eb^{2+\delta_0}(\omega)}^{2+\delta_0}e^{-\alpha(T-\tau)}+C\int_\tau^\infty e^{-\alpha(T-t)}\|H_u(t)\|_{L^{2+\delta_0}(\omega)}^{2+\delta_0}\,dt,
$$
see Corollary \ref{CorA.1} of Appendix.
\par
 Then, using \eqref{1.1$9'$} in order to estimate the integral into the right-hand side of this formulae, we get
\begin{equation}\label{1.21}
\|u\|_{W^{(1,2),2+\delta_0}_\eb(\Omega_T)}\le
Q_1(\|u_\tau\|_{V^{p}_\eb(\omega)})e^{-\alpha(T-\tau)}+
Q_1(\|g\|_{L^{2+\delta_0}_b(\Omega_+^\tau)}),
\end{equation}
where the constant $C$ and the function $Q_1$ are independent of $\eb$,
$\tau$ and $T$. We now recall that
$$
W^{(1,2),s}_\eb(\Omega_T)\subset W^{(1,2),s}_0(\Omega_T)
\equiv (W^{(1,2),s}(\Omega_T)\cap \{u\DOM=0\})
$$
and, due to the embedding theorem for anisotropic Sobolev spaces
\begin{equation}\label{1.22}
W^{(1,2),s}(\Omega_T)\subset L^{r(s)}(\Omega_T),\
\text{ where }\ \frac1{r(s)}=\frac1s-\frac2{n+2},
\end{equation}
see \cite{25}. Consequently, analogously to \eqref{1.1$9'$}, we have
\begin{equation}\label{1.23}
\|f(u)\|_{L^{2+\delta_1}(\Omega_T)}\le
Q_2(\|u_\tau\|_{V^p_\eb(\omega)})e^{-\alpha(T-\tau)}+
Q_2(\|g\|_{L^{2+\delta_0}_b(\Omega_+^\tau)}),
\end{equation}
where $\alpha>0$ and $Q_2$ are independent of $\eb$, $\tau$ and $T$ and
\begin{equation}\label{1.24}
2+\delta_1:=\frac{r(2+\delta_0)}q>\frac{r(2+\delta_0)}{q_{max}}=(2+\delta_0)
\frac{n-2}{n-2-2\delta_0}>2+\delta_0.
\end{equation}
Iterating the above procedure, we finally derive estimates
\eqref{1.21} and \eqref{1.23} with the exponent
 $2+\delta_l\equiv p$. Indeed, formulae \eqref{1.22} and \eqref{1.24}
guarantee that the number $l$ of the iterations will be finite. Thus,
estimate \eqref{1.6} is proved.
\par
In order to finish the proof of Theorem \ref{Th1.1}, there remains to note
that the existence of a solution $u\in
W^{(1,2),p}_{\eb,b}(\Omega_+^\tau)$ of problem \eqref{1.1} can be
proved in a standard way based on a priori estimate \eqref{1.6}.
Indeed, for instance, the existence of a solution of the analogue
of \eqref{1.1} in a finite cylinder
$\Omega_{\tau,N}:=(\tau,\tau+N)\times\Omega$ can be proved using
the Leray-Schauder fixed point principle and the existence of a
solution $u$ in the infinite cylinder $\Omega_+^\tau$ can be
obtained then by passing to the limit $N\to\infty$, see e.g.
\cite{26,27} for the details. Theorem \ref{Th1.1} is proved.
\end{proof}
\begin{remark}\label{Rem1.1} We note that estimate \eqref{1.6} is valid
for every $p\ge2$ although we have formally proved it only for
$p>p_{min}$. Indeed, we have used the last assumption only in order to
estimate the term $\eb^2\phi(\tau)\|f(u(\tau)\|_{L^2(\omega)}^2$ in
\eqref{1.11} which appears after applying the Schwartz inequality
to the term $\eb^2\phi(\tau)\(\Dt u(\tau),f(u(\tau))\)_{L^2(\omega)}$.
 But the growth
restriction of \eqref{1.2}, Lemma \ref{LemA.1} and the appropriate
interpolation inequality
 allow to estimate
this term in more accurate
way:
$$
 |\eb^2\(\Dt u(\tau),f(u(\tau))\)_{L^2(\omega)}|\le
 \mu\|u\|_{W^{(1,2),2}_\eb(\Omega_\tau)}^2+
Q_\mu(\|u_\tau\|_{V^2_\eb(\omega)}),
$$
where the parameter $\mu$ can be arbitrarily small and a function
$Q_\mu$ depends on $\mu$, but is independent of $\eb$ (see
\cite{23} for the details).
 Inserting this
estimate to the right-hand side of \eqref{1.11}, we can easily derive
\eqref{1.6} with $p=2$.
\end{remark}
\begin{remark}\label{Rem1.2} If we need not estimate \eqref{1.6} to be
uniform with respect to $\eb\to0$, it is possible to relax the
growth restriction \eqref{1.2}(3) till $q<q_{max}':=\frac{n+1}{n-3}$.
Indeed, in this case, it is sufficient to use the embedding
$W^{2,2}(\Omega_0)\subset L^{2q_{max}'}(\Omega_0)$ instead of
\eqref{1.18}
in the proof of Theorem \ref{Th1.1}.
\end{remark}

\begin{corollary}\label{Cor1.1} Let the assumptions of Theorem \ref{Th1.1} hold and
let $u\in W^{(1,2),p}_{\eb,b}(\Omega_+^\tau)$ be a solution of
\eqref{1.1}. Then, the following estimate holds:
\begin{equation}\label{1.25}
\|u(t)\|_{V_\eb^p(\omega)}\le Q(\|u_\tau\|_{V^p_\eb(\omega)})
e^{-\alpha(t-\tau)}+Q(\|g\|_{L^p_b(\Omega_+^\tau)}),
\end{equation}
where the constant $\alpha>0$ and the function $Q$ are indepndent of
$\eb$, $\tau$, $t\ge\tau$ and $u$.
\end{corollary}
Indeed, \eqref{1.25} is an immediate corollary of \eqref{1.6} and
the fact that $V^p_\eb(\omega)$ is the  uniform (with respect to
$\eb$) trace space of functions belonging to
$W^{(1,2),p}_{\eb,b}(\Omega_+^\tau)$, see Apendix.

\begin{corollary}\label{Cor1.2} Let the assumptions of Theorem \ref{Th1.1} hold and
let, in addition, the external forces $g$ belong to
$L^{p_1}_b(\Omega_+^\tau)$, for some $p_1>p$. Then, every solution
$u\in W^{(1,2),p}_{\eb,b}(\Omega_+^\tau)$
of problem \eqref{1.1} satisfies the following estimate:
\begin{equation}\label{1.26}
\|u\|_{W^{(1,2),p_1}_\eb(\Omega_T)}\le Q(\|u_\tau\|_{V^p_\eb(\omega)})
e^{-\alpha(T-\tau)}+Q(\|g\|_{L^{p_1}_b(\Omega_+^\tau)}),\ \
T\ge\tau+1,
\end{equation}
where the constant $\alpha>0$ and the function $Q$ are independent of
$\eb$, $\tau$, $T$ and $u$.
\end{corollary}
Indeed, since $W^{(1,2),p}(\Omega_T)\subset C(\Omega_T)$ (due to our
choice of the exponent $p$) then, estimate \eqref{1.6} implies that
\begin{equation}\label{1.27}
\|f(u)\|_{L^\infty(\Omega_T)}\le Q(\|u_\tau\|_{V^p_\eb(\omega)})
e^{-\alpha(T-\tau)}+Q(\|g\|_{L^p(\Omega_+^\tau)}),
\end{equation}
where the constant $\alpha>0$ and the function $q$ are independent of
$\eb$, $\tau$, $T$ and $u$. Rewriting now equation \eqref{1.1} in the
form of \eqref{1.20} and applying the uniform (with respect to $\eb$)
interior $L^{p_1}$-regularity estimate to this equation (see
Corollary \ref{CorA.2} and estimate \eqref{A.28}), we derive estimate
\eqref{1.26}.


\section{Uniqueness of the solutions}\label{s2}

In this section, we prove that the solution $u(t)$ of problem
\eqref{1.1} which is constructed in Theorem \ref{Th1.1} is unique if $\eb>0$
is small enough. Moreover, we also verify the differentiability of
that solution with respect to  the initial data $u_\tau\in
V^p_\eb(\omega)$ in the corresponding functional spaces. We start with
the following theorem.
\begin{theorem}\label{Th2.1} Let the assumptions of Theorem \ref{Th1.1} hold and
let, in addition, $\eb\le\eb_0:=\eb_0(a,f,\gamma)$ is small
enough. Then,  for every two solutions
$u_1(t)$ and $u_2(t)$ of problem \eqref{1.1}, the following estimate holds:
\begin{equation}\label{2.1}
\|u_1-u_2\|_{W^{(1,2),p}_\eb(\Omega_T)}\le Ce^{\Lambda_0(T-\tau)}
\|u_1(\tau)-u_2(\tau)\|_{V^p_\eb(\omega)},
\end{equation}
where the constant $\Lambda_0$ is independent of
$\eb\le\eb_0$, $\tau\in\R$, $T\ge\tau$, $u_1$ and $u_2$
and the constant $C$ depends on $\|u_i(\tau)\|_{V^p_\eb(\omega)}$,
but is independent of $\eb$, $\tau$ and $T$. In
particular, the solution of \eqref{1.1} is unique if $\eb\le\eb_0$.
\end{theorem}
\begin{proof} We set $v(t):=u_1(t)-u_2(t)$. Then, this function
satisfies the following equation:
\begin{equation}\label{2.2}
a(\eb^2\Dt^2 v+\Dx v)-\gamma\Dt v-l(t)v=0,\ \ v\DOM=0,\ \
v\big|_{t=\tau}=u_1(\tau)-u_2(\tau),
\end{equation}
where $l(t)=l(t,x):=\int^1_0f'(su_1(t)+(1-s)u_2(t))\,dt$. Moreover, due to
the second assumption of \eqref{1.2}, estimate \eqref{1.6} and the embedding
$W^{(1,2),p}(\Omega_T)\subset C(\Omega_T)$, we have
\begin{equation}\label{2.3}
l(t,x)\ge-K,\ \ \|l(t,x)\|_{L^\infty(\Omega_+^\tau)}\le M,
\end{equation}
where the constant $K$ is defined in \eqref{1.2} and the constant
$M$ depends on the norms $\|u_i(\tau)\|_{V^p_\eb(\omega)}$, $i=1,2$,
and $\|g\|_{L^p_b(\Omega_+^\tau)}$,
 but is
independent of $\eb$ and $\tau$. It is however convenient to consider
more general (than \eqref{2.2}) problem
\begin{equation}\label{2.4}
a(\eb^2\Dt^2w+\Dx w)-\gamma\Dt w-l(t)w=h(t), \ w\DOM=0,\ \
 w\big|_{t=\tau}=w_\tau,
\end{equation}
where the given matrix-valued function $l(t)$ satisfies \eqref{2.3}
and $h(t)=h(t,x)$ are given external forces.
\begin{lemma}\label{Lem2.1} Let $\Lambda_0$ be a nonnegative number which
satisfies the following condition:
\begin{equation}\label{2.5}
\Lambda_0\gamma-\eb^2\Lambda_0^2(a_+-2a_-(a_+)^{-1}a_-)-K\ge 0,
\end{equation}
where $a_+:=1/2(a+a^*)$ and $a_-=1/2(a-a^*)$. Then, for every
$w_\tau\subset V^p_\eb(\omega)$ and every external forces $h$
satisfying
\begin{equation}\label{2.6}
e^{-\Lambda_0 t}h(t)\in L^p_b(\Omega_+^\tau),
\end{equation}
problem \eqref{2.4} has a unique solution $w(t)$ belonging to the
class
\begin{equation}\label{2.7}
e^{-\Lambda_0 t}w(t)\in W^{(1,2),p}_{\eb,b}(\Omega_+^\tau)
\end{equation}
and the following estimate holds:
\begin{multline}\label{2.8}
\|w\|_{W^{(1,2),p}_\eb(\Omega_T)}^p\le C\|w_\tau\|_{V^p_\eb(\omega)}^p
e^{p(\Lambda_0-\alpha)(T-\tau)}+\\+C\int_\tau^\infty
 e^{-p\alpha|T-t|+p\Lambda_0(T-t)}\|h(t)\|_{L^p(\omega)}^p\,dt,
\end{multline}
where the positive constants $\alpha$ and $C$ depend on $M$ and
$\Lambda_0$, but are independent of $\eb$, $\tau$ and $T\ge\tau$.
\end{lemma}
\begin{proof} We first note that, due to the fact that
$V^p_\eb(\omega)$ is the uniform (with respect to $\eb$) trace space
for functions belonging to $W^{(1,2),p}_\eb(\Omega_\tau)$, it is
sufficient to verify Lemma \ref{Lem2.1} for the case $w_\tau=0$ only. In order
to do so, we set $\theta(t):=e^{-\Lambda_0t}w(t)$. Then, this function
belongs to $W^{(1,2),p}_{\eb,b}(\Omega_+^\tau)$ (due to assumption
 \eqref{2.7}) and satisfies the following equation:
\begin{multline}\label{2.9}
a(\eb^2\Dt \theta+\Dx\theta)-(\gamma-2\eb^2\Lambda_0
a)\Dt\theta-\\- (\Lambda_0\gamma-\eb^2\Lambda_0^2a+l(t))\theta=
\tilde h(t):=e^{-\Lambda_0t}h(t), \
\theta\DOM=0,\;\theta\big|_{t=\tau}=0.
\end{multline}
Multiplying now this equation by $\phi_T(t)\theta(t)$ (where the
weight function $\phi_T(t):=e^{-\alpha|T-t|}$) and integrating over
$\Omega_+^\tau$, we obtain after the standard transformations
(integrating by parts and using that $\gamma=\gamma^*$, $l(t)\ge-K$
and estimate \eqref{1.8}) that
\begin{multline}\label{2.10}
\eb^2\<a_+\Dt \theta.\Dt\theta,\phi_T\>_\tau+
\<a_+\Nx \theta.\Nx\theta,\phi_T\>_\tau+
\<(\Lambda_0\gamma-\eb^2\Lambda_0^2a_+-K)v.v,\phi_T\>_\tau
\\\le|\<\tilde h,\phi_T\theta\>_\tau|+
C\eb^2\alpha\<|\Dt \theta|^2+|\theta|^2,\phi_T\>_\tau+
2\eb^2\Lambda_0|\<a_-\Dt\theta.\theta,\phi_T\>_\tau|,
\end{multline}
where the constant $C$ depends only on $a$ and $\gamma$
and $\<\cdot,\cdot\>_\tau$ stands for the inner product in
$L^2(\Omega^\tau_+)$. Estimating
the last term in the right-hand side of \eqref{2.10} as follows:
$$
2\eb^2\Lambda_0|a_-\Dt \theta.\theta|\le
1/2\eb^2\Lambda_0^2a_+\Dt\theta.\Dt\theta-2\eb^2a_-(a_+)^{-1}a_-\theta.\theta,
$$
fixing the parameter $\alpha>0$ to be small enough
and using the Friedrichs and Schwartz inequalities,
we have
\begin{multline*}
\eb^2\<a_+\Dt \theta.\Dt\theta,\phi_T\>_\tau+
\<a_+\Nx \theta.\Nx\theta,\phi_T\>_\tau+\\+
4\<(\Lambda_0\gamma-\eb^2\Lambda_0^2(a_+-2a_-(a_+)^{-1}a_-)-K)v.v,\phi_T\>_\tau
\le C_1\<|\tilde h|^2,\phi_T\>_\tau.
\end{multline*}
Using assumption \eqref{2.5}, positivity of $a_+$ and the obvious
inequality $\phi_T(t)\ge e^{-\alpha}$ for $t\in[T,T+1]$,
 we have
$$
\eb^2\|\Dt\theta\|_{L^2(\Omega_T)}^2+\|\Nx w\|^2_{L^2(\Omega_T)}\le
C_2\< |\tilde h|^2,\phi_T\>_\tau
$$
Returning to the variable $w(t)=e^{\Lambda_0t}\theta(t)$, we derive
\begin{equation}\label{2.11}
\eb^2\|\Dt w\|_{L^2(\Omega_T)}^2+\|\Nx w\|_{L^2(\Omega_T)}^2\le
C_3\int^\infty_\tau
e^{-\alpha|T-t|+2\Lambda_0(T-t)}\|h(t)\|_{L^2(\omega)}^2\,dt,
\end{equation}
Estimate \eqref{2.8} (with $w_\tau=0$) can be now derived from
\eqref{2.11} iterating the maximal elliptic regularity estimate
\eqref{A.27} exactly as in the end of the proof of Theorem \ref{Th1.1}.
The existence of the solution can be then verified in a standard
way based on a priori estimate \eqref{2.8}, see e.g. \cite{26,27}.
 Lemma \ref{Lem2.1} is proved.
\end{proof}
We are now ready to finish the proof of Theorem \ref{Th2.1}. To this end, we
note that the left-hand side of \eqref{2.5} tends to
$\Lambda_0\gamma-K$ as $\eb\to0$ and, consequently, for every
sufficiently large $\Lambda_0>0$, we may fix (due to positivity of
the matrix $\gamma$)
 $\eb_0=\eb_0(\Lambda_0,K,a,\gamma)$ such that \eqref{2.5} is
satisfied, for every $\eb\le\eb_0$. Applying then estimate \eqref{2.8}
(with $h\equiv0$) to
equation \eqref{2.2}, we finish the proof of Theorem \ref{Th2.1}.
\end{proof}
Let us assume from now on that
\begin{equation}\label{2.13}
g\in L^p_b(\Omega),\ \ \text{where}\  \ \Omega:=\R\times\omega.
\end{equation}
Then, under the assumptions of Theorem \ref{Th2.1}, problem \eqref{1.1}
defines a two-para\-met\-ri\-cal family of solving operators
 $\{U^\eb_g(t,\tau),\,
\tau\in\R,\,t\ge\tau\}$ via
\begin{equation}\label{2.14}
U^\eb_g(t,\tau):V^p_\eb(\omega)\to V^p_\eb(\omega),\ \
u(t):=U^\eb_g(t,\tau)u_\tau,
\end{equation}
where $u(t)$ solves \eqref{1.1} and $u(\tau)=u_\tau$ which, obviously,
generates a dynamical process on $V^p_\eb(\omega)$, i.e.
\begin{equation}\label{2.15}
U^\eb_g(t,\tau_1)\circ U^\eb_g(\tau_1,\tau)=U^\eb_g(t,\tau),\
\ t\ge\tau_1\ge\tau\in\R.
\end{equation}
Moreover, Theorem \ref{Th2.1} shows that these operators are uniformly (with
respect to~$\eb$) Lipschitz continuous in $V^p_\eb(\omega)$. Our next
task is to prove their Frechet differentiability with respect to the
initial data $u_\tau\in V^p_\eb(\omega)$.  To this end, we consider
the following formal equation of variations associated with a solution
$u(t):=U^\eb_g(t,\tau)u_\tau$:
\begin{equation}\label{2.16}
a(\eb^2\Dt^2v+\Dx v)-\gamma\Dt v-f'(u(t))v=0,\ \ v\DOM=0, \ \
v\big|_{t=\tau}=v_\tau.
\end{equation}
Then, due to Lemma \ref{Lem2.1}, we have
\begin{equation}\label{2.17}
\|v(t)\|_{V^p_\eb(\omega)}\le C\|v_\tau\|_{V^p_\eb(\omega)}
e^{(\Lambda_0-\alpha)(t-\tau)},
\end{equation}
where the solution $v(t)$ satisfies \eqref{2.7} and the constants
$\alpha>0$ and $C$ are independent of $\eb$, $\tau$ and $T$. The
following theorem shows that \eqref{2.16} defines indeed the Frechet
derivative of the process $U^\eb_g(t,\tau)$ at $u_\tau$.
\begin{theorem}\label{Th2.2} Let the assumptions of Theorem \ref{Th2.1} hold. Let
also $u(t)$ and $u_1(t)$ be two solutions of \eqref{1.1} and
$v(t)$ be a solution of \eqref{2.16} with
$v_\tau:=u(\tau)-u_1(\tau)$ (associated with $u(t)$). Then, there
exists $\eb_0'=\eb_0'(f,a,\gamma)>0$ such that $\eb_0'\le\eb_0$ and
for every $\eb\le\eb_0'$ the following estimate is valid:
\begin{equation}\label{2.18}
\|u(t)-u_1(t)-v(t)\|_{W^{(1,2),p}_\eb(\Omega_T)}\le
Ce^{(2\Lambda_0-\alpha)(T-\tau)}\|u(\tau)-u_1(\tau)\|_{V^p_\eb(\omega)}^2,
\end{equation}
where the constants $C$ and $\alpha>0$ depend on
$\|u(\tau)\|_{V^p_\eb(\omega)}$ and $\|u_1(\tau)\|_{V^p_\eb(\Omega)}$,
but are independent of $\eb$, $\tau$ and $T$.
\end{theorem}
\begin{proof} We set $w(t):=u(t)-u_1(t)-v(t)$. Then, this function
satisfies the following equation:
\begin{equation}\label{2.19}
a(\eb^2\Dt^2w+\Dx w)-\gamma\Dt w-f'(u(t))w=h_{u,u_1}(t), \ \ w\DOM=0,\ \
w\big|_{t=\tau}=0,
\end{equation}
where $h_{u,u_1}(t):= \int^1_0[f'(u(t))-f'(u(t)+s(u_1(t)-u(t)))]\,ds
(u(t)-u_1(t))$. Moreover, since $V_0^p(\omega)\subset C(\omega)$ and
$f\in C^2$, then estimates \eqref{1.6} and \eqref{2.1} implies that
\begin{equation}\label{2.20}
\|h_{u,u_1}(t)\|_{L^\infty(\omega)}\le
C\|u(t)-u_1(t)\|_{L^\infty(\omega)}^2\le
C_1\|u(\tau)-u_1(\tau)\|_{V^p_\eb(\omega)}^2e^{2(\Lambda_0-\alpha)(t-\tau)},
\end{equation}
where the constants $C_i$ depend on $\|u(\tau)\|_{V^p_\eb(\omega)}$
and $\|u_1(\tau)\|_{V^p_\eb(\omega)}$, but are independent of $\eb$,
$\tau$ and $t\ge\tau$. Fixing now $\eb_0'>0$ small enough that
assumption \eqref{2.5} holds with $\Lambda_0$ replaced by
$2\Lambda_0$ and applying Lemma \ref{Lem2.1} (with $2\Lambda_0$ instead of
$\Lambda_0$) to equation \eqref{2.19}, we derive estimate
\eqref{2.18} and finish the proof of Theorem \ref{Th2.2}.
\end{proof}
\begin{corollary}\label{Cor2.1} Let the assumptions of Theorem \ref{Th2.2} hold
and let $u(t)$, $u_1(t)$ and $v(t)$ be the same as in Theorem
\ref{Th2.2}. Then, for every $q\ge p$, $q<\infty$ and every $T\ge\tau+1$,
 the following estimate
holds:
\begin{equation}\label{2.21}
\|u(t)-u_1(t)-v(t)\|_{W^{(1,2),q}_\eb(\Omega_T)}\le
C_qe^{(2\Lambda_0-\alpha)(T-\tau)}\|u(\tau)-u_1(\tau)\|_{V^p_\eb(\omega)}^2,
\end{equation}
where the constants $C_q$ and $\alpha>0$ depend on
$\|u(\tau)\|_{V^p_\eb(\omega)}$, $\|u_1(\tau)\|_{V^p_\eb(\omega)}$ and $q$,
but are independent of $\eb$, $\tau$ and $T$.
\end{corollary}
Indeed, rewriting equation \eqref{2.19} in the form
\begin{equation}\label{2.22}
a(\eb^2\Dt^2 w+\Dx w)-\gamma\Dt w=f'(u(t))w(t)+h_{u,u_1}(t),\ \
w\DOM=w\big|_{t=\tau}=0,
\end{equation}
applying the interior regularity estimate \eqref{A.28} (where the
exponent $p$ is replaced by $q$) and using estimates \eqref{2.18} and
\eqref{2.20} for estimating the right-hand side of \eqref{2.22},
we derive estimate \eqref{2.21}.

 \begin{corollary}\label{Cor2.2} Let the assumptions of Theorem \ref{Th2.2}
 hold. Then, the operators $U^\eb_g(t,\tau)$ are Frechet
 differentiable with respect to the initial data, their Frechet
 derivative is defined by $D_uU^\eb_g(t,\tau)(u_\tau)\xi:=v(t)$,
where $v(t)$ is the solution of \eqref{2.16} with $v_\tau=\xi$,
and the following estimates hold:
\begin{multline}\label{2.23}
\|U^\eb_g(t,\tau)u_\tau^1-U^\eb_g(t,\tau)u_\tau^2-D_u
U^\eb_g(t,\tau)(u^1_\tau)(u_\tau^1-u_\tau^2)\|_{V^p_\eb(\omega)}\le\\\le
 Ce^{2\Lambda_0(t-\tau)}\|u_\tau^1-u_\tau^2\|_{V^p_\eb(\omega)}^2
\end{multline}
for every $u^1_\tau,u^2_\tau\in V^p_\eb(\omega)$ and, consequently
\begin{equation}\label{2.24}
\|D_uU^\eb_g(t,\tau)(u^1_\tau)-D_uU^\eb_g(t,\tau)(u^2_\tau)
\|_{\Cal L(V^p_\eb(\omega),V^p_\eb(\omega))}\le
Ce^{2\Lambda_0(t-\tau)}\|u^1_\tau-u^2_\tau\|_{V^p_\eb(\omega)},
\end{equation}
where the constant $C$ depends on $\|u_\tau^1\|_{V^p_\eb(\omega)}$,
$\|u^2_\tau\|_{V^p_\eb(\omega)}$ and $\|g\|_{L^p_b}$, but is
independent of $\eb$, $\tau$ and $t$.
\end{corollary}
Indeed, estimate \eqref{2.23} is an immediate corollary of
\eqref{2.18} and estimate \eqref{2.24} is a standard corollary of
\eqref{2.23}.
\par
Arguing analogously, but using estimate \eqref{2.21} instead of
\eqref{2.18}, we derive the following result.
\begin{corollary}\label{Cor2.3} Under the assumptions of Corollary \ref{Cor2.2}
the following estimates hold, for every $q\ge p$ and $T\ge\tau+1$:
\begin{multline}\label{2.25}
\|U^\eb_g(t,\tau)u_\tau^1-U^\eb_g(t,\tau)u_\tau^2-D_u
U^\eb_g(t,\tau)(u^1_\tau)(u_\tau^1-u_\tau^2)\|_{V^q_\eb(\omega)}\le\\\le
 C_qe^{2\Lambda_0(t-\tau)}\|u_\tau^1-u_\tau^2\|_{V^p_\eb(\omega)}^2
\end{multline}
for every $u^1_\tau,u^2_\tau\in V^p_\eb(\omega)$ and, consequently
\begin{equation}\label{2.26}
\|D_uU^\eb_g(t,\tau)(u^1_\tau)-D_uU^\eb_g(t,\tau)(u^2_\tau)
\|_{\Cal L(V^p_\eb(\omega),V^q_\eb(\omega))}\le
C_qe^{2\Lambda_0(t-\tau)}\|u^1_\tau-u^2_\tau\|_{V^p_\eb(\omega)},
\end{equation}
where the constant $C_q$ depends on $q$, $\|u_\tau^1\|_{V^p_\eb(\omega)}$,
$\|u^2_\tau\|_{V^p_\eb(\omega)}$ and $\|g\|_{L^p_b}$, but is
independent of $\eb$, $\tau$ and $t$.
\end{corollary}

We now recall that, for $\eb>0$, operators $U^\eb_g(t,\tau)$ are
defined on the space $V^p_\eb(\omega)\sim W^{2-1/p,p}(\omega)$
and, for $\eb=0$, the limit process $U^0_g(t,\tau)$ is
 defined on the different space
$V^p_0(\omega)\sim W^{2(1-1/p),p}(\omega)\ne V^p_\eb(\omega)$ which
is not convenient for the study of the limit $\eb\to0$.
 In order to
overcome this difficulty,
we consider
the following discrete analogue of process \eqref{2.14}:
\begin{equation}\label{2.27}
U^\eb_g(l,m):V^p_\eb(\omega)\to V^p_\eb(\omega),\ \ l,m\in\Bbb Z,\ \
l\ge m.
\end{equation}
 Moreover,
 we assume, in addition, that the exponent $p$ satisfies
$p\ge 2p_{min}$ and use the following obvious embeddings:
\begin{equation}\label{2.28}
 V^p_\eb(\omega)\subset  V_0^p(\omega)\subset V^{p/2}_\eb(\omega),
\end{equation}
which are, in a fact, uniform with respect to $\eb$, see Definition \ref{Def1.1}.
Then, we have $p/2>p_{min}$ and, consequently, all previous results
remain true if we replace $p$ by $p/2$. In particular, Theorem
\ref{Th1.1}, Corollary \ref{Cor1.2} and embeddings \eqref{2.28} imply that
\begin{multline}\label{2.29}
\|U^\eb_g(l,m)u_m\|_{V_0^p(\omega)}\le
\|U^\eb_g(l,m)u_m\|_{V^p_\eb(\omega)}\le
Q(\|u_m\|_{V^{p/2}_\eb(\omega)})e^{-\alpha
(l-m)}+\\+Q(\|g\|_{L^p_b(\Omega)})\le
Q(2\|u_m\|_{V_0^p(\omega)})e^{-\alpha(l-m)}+Q(\|g\|_{L^p_b(\Omega)}),
\end{multline}
for every $u_m\in V^p_0(\omega)$, and the constant $\alpha>0$ and
the monotonic function $Q$ are independent of $0\le\eb\le\eb_0'$,
$l,m\in\Bbb Z$ and $l\ge m$. Moreover, using Corollaries \ref{Cor2.2}-\ref{Cor2.3} and
arguing analogously, we derive the following result.
\begin{corollary}\label{Cor2.4} Let the assumptions of Theorem \ref{Th2.2} hold and
let, in addition, $p>2p_{min}$. Then
the following estimates hold:
\begin{multline}\label{2.30}
\|U^\eb_{g}(l,m)u_m^1-U^\eb_{g}(l,m)u_m^2-D_u
U^\eb_{g}(l,m)(u^1_m)(u_m^1-u_m^2)\|_{V^p_0(\omega)}\le\\\le
 Ce^{-2\Lambda_0(l-m)}\|u_m^1-u_m^2\|_{V^p_0(\omega)}^2,
\end{multline}
for every $u^1_m,u^2_m\in V^p_0(\omega)$ and $l,m\in\Bbb Z$, $l\ge m$,
  consequently
\begin{equation}\label{2.31}
\|D_uU^\eb_{g}(l,m)(u^1_m)-D_uU^\eb_{g}(l,m)(u^1_m)
\|_{\Cal L(V^p_0(\omega),V^p_0(\omega))}\le
C_qe^{2\Lambda_0(l-m)}\|u^1_m-u^2_m\|_{V^p_0(\omega)},
\end{equation}
where the constant $C$ depends on $\|u_m^1\|_{V^p_0(\omega)}$,
$\|u^2_m\|_{V^p_0(\omega)}$ and $\|g\|_{L^p_b}$, but is
independent of $\eb$, $l$ and $m$.
\end{corollary}
Thus, in contrast to the continuous dynamics $\{U^\eb_g(t,\tau),\,
 \tau\in\R,\,t\ge\tau\}$, discrete cascades \eqref{2.27} are well
defined on the space $V^p_0(\omega)$ which is independent of $\eb$.
\par
To conclude, we formulate the result on injectivity of
operators $U^\eb_g(t,\tau)$.
\begin{theorem}\label{Th2.3} Let the assumptions of Theorem \ref{Th2.1} hold and
let
$$
U^\eb_g(t,\tau)u_\tau^1=U^\eb_g(t,\tau)u_\tau^2,
$$
for some $\tau\in\R$, $t\ge\tau$ and $u_\tau^1,u^2_\tau\in
V^p_\eb(\omega)$. Then, necessarily, $u_\tau^1=u_\tau^2$.
\end{theorem}
The proof of this Theorem is based on the logarithmic convexity
results (see \cite{1}) for solutions of \eqref{1.1} and can be
found, e.g. in \cite{27}.


\section{Uniform attractors and their convergence as $\eb\to0$}\label{s3}

In this section, we construct the global attractors $\Cal A_\eb$
for problems \eqref{1.1} and investigate their behavior as
$\eb\to0$. Since the external forces $g(t)$ in \eqref{1.1} (which
are assumed from now on to be defined on the whole cylinder
$\Omega$ and to belong to the space $L^p_b(\Omega)$) depend
explicitly on $t$, then we use below the skew-product technique in
order to reduce the nonautonomous dynamical process \eqref{2.14}
associated with problem \eqref{1.1} to the autonomous semigroup on
the extended phase space. Following the general procedure (see
\cite{9} and \cite{14}), we define a hull $\Cal H(g)$ of the
external forces $g$ as follows:
\begin{equation}\label{3.1}
\Cal H(g):=\big[T_hg,\,h\in\R\big]_{L^p_{loc,w}(\Omega)},\ \
(T_hg)(t):=g(t+h).
\end{equation}
Here $[\cdot]_{L^p_{loc,w}(\Omega)}$ stands for the closure in the
space $L^p_{loc,w}(\Omega)$ which is the space $L^p_{loc}(\Omega)$
endowed by the weak topology. We recall  that a sequence $g_k\to g$ in
$L^p_{loc,w}(\Omega)$ as $k\to\infty$ if and only if
 $g_k\big|_{\Omega_T}\to g\big|_{\Omega_T}$ weakly in $L^p(\Omega_T)$,
for every $T\in\R$. It is also well-known, that every bounded
subset of $L^p_{loc,w}(\Omega)$ is  precompact and metrizable
 and,
consequently (due to the assumption $g\in L^p_b(\Omega)$), hull
\eqref{3.1} is a compact metrizable subset of $L^p_{loc,w}(\Omega)$.
Thus, a function $\xi(t)$ belongs to $\Cal H(g)$ if and only if there
exists a sequence $\{h_n\}_{n=1}^\infty\in\R$ such that
\begin{equation}\label{3.2}
\xi=\lim_{n\to\infty}T_{h_n}g\ \  \text{in the space
$L^p_{loc,w}(\Omega)$.}
\end{equation}
Moreover, it also obvious that the group $\{T_h,\, h\in\R\}$ of
temporal translations acts on $\Cal H(g)$, i.e.
\begin{equation}\label{3.3}
T_h:\Cal H(g)\to \Cal H(g),\ \ T_h\Cal H(g)=\Cal H(g),\ \ h\in\R.
\end{equation}
In order to construct the attractor of \eqref{1.1}, we consider
the following family of equations of type \eqref{1.1} which correspond to
all external forces $\xi\in\Cal H(g)$:
\begin{equation}\label{3.4}
a(\eb^2\Dt^2 u+\Dx u)-\gamma\Dt u-f(u)=\xi(t),\ \
u\big|_{t=\tau}=u_\tau, \ \ u\DOM=0,\ \xi\in\Cal H(g)
\end{equation}
which generates the family $\{U_\xi^\eb(t,\tau),\ \ \xi\in\Cal H(g)\}$
of dynamical processes in $V^p_\eb(\omega)$ (under the assumptions of
Theorem \ref{Th2.1}). This family of processes generates a semigroup
$\{\Bbb S_t^\eb,\,t\ge0\}$ on the extended phase space
$\Phi_\eb:=V^p_{\eb,w}(\omega)\times \Cal H(g)$
(as usual $V^p_{\eb,w}(\omega)$ denotes the space $V^p_\eb(\omega)$
endowed by the weak topology)
 by the following
 expression:
\begin{equation}\label{3.5}
\Bbb S_t^\eb(u_0,\xi):=(U_\xi^\eb(t,0)u_0,T_t\xi),\ \
\Bbb S_t^\eb:\Phi_\eb\to\Phi_\eb,\ t\ge0,\ (u_0,\xi)\in\Phi_\eb
\end{equation}
(see \cite{9} for the details). Thus, we describe the 'longtime'
behavior of solutions of \eqref{3.4} in terms of the global
attractor of semigroup \eqref{3.5} in the extended phase space
$\Phi_\eb$. For the convenience of the reader, we recall below the
definition of the attractor adapted to our case,
 see e.g. \cite{4}, \cite{9} and \cite{24}
for the detailed exposition.
\begin{definition}\label{Def3.1} A set $\Bbb A_\eb\subset\Phi_\eb$ is a global
attractor for the semigroup $\Bbb S_t^\eb$ if the following conditions
are satisfied:
\par
1.\ The set $\Bbb A_\eb$ is compact in $\Phi_\eb$.
\par
2.\ This set is strictly invariant with respect to $\Bbb S_t$, i.e.
$\Bbb S_t\Bbb A_\eb=\Bbb A_\eb$.
\par
3. For every bounded subset $\Bbb B\subset\Phi_\eb$
and every neighborhood $\Cal O(\Bbb A_\eb)$ of the set $\Bbb A_\eb$ in the
topology of $\Phi_\eb$,
 there exists $T=T(\Bbb B,\Cal O)$ such that
\begin{equation}\label{3.6}
\Bbb S_t^\eb\Bbb B\subset\Cal O(\Bbb A_\eb),\ \text{ for }\ \ t\ge T.
\end{equation}
A projection $\Cal A_\eb:=\Pi_1\Bbb A_\eb$ of the global attractor
$\Bbb A_\eb$ to the first component is called a {\it uniform}
attractor of family  \eqref{3.4}.
\end{definition}
The next theorem establishes the existence of the attractor described
above.
\begin{theorem}\label{Th3.1} Let the assumptions of Theorem \ref{Th2.1} hold and
let $g\in L^p_b(\Omega)$. Then, semigroup \eqref{3.5} possesses a global
attractor $\Bbb A_\eb$ in the phase space $\Phi_\eb$ and,
consequently, family of problems \eqref{3.4} possesses a uniform
attractor $\Cal A_\eb$ which can be described as follows:
\begin{equation}\label{3.7}
\Cal A_\eb=\Pi_0\cup_{\xi\in\Cal H(g)}\Cal K^\eb_\xi,
\end{equation}
where $\Cal K^\eb_\xi$ is a set of all solutions of problem
\eqref{3.4} (with the right-hand side $\xi\in\Cal H(g)$) which are
defined for all $t\in\R$ and belong to $C_b(\R,V^p_\eb(\omega))$ and
$\Pi_0u:=u(0)$.
\end{theorem}
\begin{proof} According to the abstract theorem on the global (and
uniform) attractors existence (see \cite{4}, \cite{9} and
\cite{24}),
 it is sufficient to verify the following conditions:
\par
1.\ The semigroup $\Bbb S_t^\eb$ possesses a compact absorbing set $\Bbb
  B$ in $\Phi_\eb$.
\par
2. The operators $\Bbb S_t^\eb$ are continuous on $\Bbb B$, for every
fixed $t\ge0$.
\par
Let us verify these conditions. It follows from estimate \eqref{1.25}
that the set
\begin{equation}\label{3.8}
\Bbb B:=\{(u_0,\xi)\in\Phi_\eb,\ \|u_0\|_{V^p_\eb(\omega)}\le
 2Q(\|g\|_{L^p_b(\Omega)}),\ \xi\in\Cal H(g)\}
\end{equation}
is an absorbing set for the semigroup $\Bbb S_t^\eb$ (here we have
implicitly used the obvious fact that $\|\xi\|_{L^p_b(\Omega)}\le
 \|g\|_{L^p_b(\Omega)}$, for every $\xi\in\Cal H(g)$). Moreover,
since the space $V^p_\eb(\omega)$ is reflexive, then bounded subsets
of it are precompact in a weak topology. Using the fact that
$\Cal H(g)$ is also compact, we derive that set \eqref{3.8} is
compact in $\Phi_\eb$. Thus, the first condition is verified.
\par
In order to verify the second one, we first note that the set $\Bbb B$
is metrizable, consequently, it is sufficient to verify only the
sequential continuity of $\Bbb S^\eb_t$ on $\Bbb B$. Indeed, let
$(u_0^n,\xi_n)\in\Bbb B$ be an arbitrary (weakly) convergent sequence in $\Bbb
B$ and let $(u_0,\xi_0)\in\Bbb B$ be its (weak) limit. We set
 $u_n(t):=U^\eb_{\xi_n}(t,0)u_0^n$. Then, by definition, these
functions satisfy the equations:
\begin{equation}\label{3.9}
a(\eb^2\Dt^2u_n(t)\!+\!\Dx u_n(t))\!-\!\gamma
\Dt u_n(t)\!-\!f(u_n(t))\!=\!\xi_n(t),\
u_n\big|_{t=0}=u^n_0,\ u_n\DOM=0.
\end{equation}
In order to verify the desired continuity, we need to prove that
$u_n(t)\to u_0(t)$ weakly in $V^p_\eb(\omega)$, for every $t\ge0$,
where $u_0(t):=U^\eb_\xi(t,0)u_0$ is a solution of the limit (as
$n\to\infty$) equation of \eqref{3.9}. We note that the
sequence $u_0^n$ is uniformly bounded in $V^p_\eb(\omega)$ (since it
converges weakly to $u_0$), consequently, due to Theorem \ref{Th1.1}, we
have
\begin{equation}\label{3.10}
\|u_n\|_{W^{(1,2),p}_\eb(\Omega_T)}\le C,
\end{equation}
where $C$ is independent of $T\ge0$ and $n\in\Bbb N$. Therefore,
the sequence of the solutions $u_n(t)$ is precompact in a weak
topology of the space $W^{(1,2),p}_{\eb,loc}(\Omega_+^0)$ (since
this space is reflexive). Let $\bar u:=\bar u(t)\in
W^{(1,2),p}_{\eb,loc}(\Omega_+^0)$ be an arbitrary
limit point of this sequence. Then, due to estimate \eqref{3.10},
the function $\bar u(t)$ belongs to $W^{(1,2),p}_{\eb,b}(\Omega_+^0)$.
Moreover, due to compactness of the embedding
$W^{(1,2),p}_\eb(\Omega_T)\subset C(\Omega_T)$, we have
\begin{equation}\label{3.11}
u_{n_k}\to \bar u,\ \ \text{ strongly in }\ \ C(\Omega_T), \ \
T\in\R_+,
\end{equation}
for the appropriate subsequence $\{n_k\}_{k=1}^\infty\in\Bbb
N$. Passing now to the limit $k\to\infty$ in equations \eqref{3.9}
and using \eqref{3.11} and that $\xi_n\to\xi$ weakly in $L^p_{loc}(\Omega)$,
we derive that $\bar u$ is a bounded solution
of the limit equation of \eqref{3.9}. Since, due to Theorem \ref{Th2.1}, this
solution is unique, then, necessarily,
 $\bar u(t)\equiv u_0(t):=U^\eb_\xi(t,0)u_0$. Moreover, since the
limit point $\bar u$ is arbitrary, then we have proved that $u_n\to
u_0$ weakly in $W^{(1,2),p}_\eb(\Omega_T)$, for every $T\in\R_+$ and,
consequently, $u_n(t)\to u_0(t)$ weakly in $V^p_\eb(\omega)$, for
every $t\in\R_+$. Thus, the second condition of the abstract theorem
on the attractors existence is also verified and, therefore, according
to this theorem, the semigroup $\Bbb S_t^\eb$ possesses indeed the
global attractor $\Bbb A_\eb$ in $\Phi_\eb$ and the family of problems
\eqref{3.4} possesses the uniform attractor $\Cal A_\eb:=\Pi_1\Bbb
A_\eb\in V^p_\eb(\omega)$. Description \eqref{3.7} is also a
standard corollary of that theorem, see \cite{4} and \cite{9}.
Theorem \ref{Th3.1} is proved.
\end{proof}
\begin{remark}\label{Rem3.1} There exists an alternative way to introduce
the concept of the uniform attractor of equation \eqref{1.1} without
using the skew-product flow on the extended phase space $\Phi_\eb$.
Namely, the set $\Cal A_\eb$ is a uniform attractor for equation
\eqref{1.1} if the following conditions are satisfied:
\par
1.\ The set $\Cal A_\eb$ is compact in $V^p_\eb(\omega)$.
\par
2. For every bounded subset $B\subset V^p_\eb(\omega)$ and
every neighborhood $\Cal O(\Cal A_\eb)$ of $\Cal A_\eb$ in a weak
topology of $V^p_\eb(\omega)$ there exists $T=T(B,\Cal O)$ such that
\begin{equation}\label{3.12}
U^\eb_g(t+\tau,\tau)B\subset\Cal O(\Cal A_\eb),\ \ \text{for every
$\tau\in\R$ and $t\ge T$}.
\end{equation}
\par
3.\ The set $\Cal A_\eb$ is a minimal set which satisfies 1) and 2).
\par
The equivalence of this definition to Definition \ref{Def3.1} is proved
in \cite{9}.
\end{remark}
\begin{remark}\label{Rem3.2} If the initial external forces $g$ satisfy the
additional assumption
\begin{equation}\label{3.13}
\Cal H(g)\ \ \text{is compact in a {\it strong} topology
 of $L^p_{loc}(\Omega)$},
\end{equation}
then, arguing in a standard way (see, e.g. \cite{9} and
\cite{26}),
 we can prove that the attractor $\Bbb A_\eb$ attracts the
bounded subsets of $\Phi_\eb$ not only in a {\it weak} topology,
but also in more natural {\it strong} topology
of $\Phi_\eb$ and
$\Cal A_\eb$ is compact in a strong topology of $V^p_\eb(\omega)$.
Nevertheless, we prefer to use the weak topology in Definition \ref{Def3.1},
since the choice of the weak topology is more convenient for what
follows.
\end{remark}
\begin{remark}\label{Rem3.3} Since the embeddings
$V^p_\eb(\omega)\subset V^{p-\delta}_\eb(\omega)$, $\delta>0$ and
$V^p_\eb(\omega)\subset C(\omega)$ are compact, then  \eqref{3.12}
implies the following convergence:
\begin{equation}\label{3.14}
\lim_{t\to\infty}\sup_{\tau\in\R}
\dist_{V^{p-\delta}_\eb(\omega)\cap
C(\omega)}\(U^\eb_g(t+\tau,\tau)B,\Cal A_\eb\)=0,
\end{equation}
for every bounded set $B\subset V^p_\eb(\omega)$ and every
$\delta>0$. Here an below $\dist_V(X,Y)$ denotes the nonsymmetric
Hausdorff distance between sets $X$ and $Y$ in the space~$V$:
\begin{equation}\label{3.15}
\dist_V(X,Y):=\sup\nolimits_{x\in X}\inf\nolimits_{y\in Y}\|x-y\|_{V}.
\end{equation}
\end{remark}
The rest  of this section is devoted to study the behavior of the
attractors $\Cal A_\eb$ as $\eb\to0$. To this end, keeping in mind
 equation \eqref{0.4}, it is convenient to consider  slightly more
general family of equations of the form \eqref{1.1}:
\begin{equation}\label{3.16}
a(\eb^2\Dt^2u+\Dx u)-\gamma\Dt u-f(u)=g_\eb(t),\ u\DOM=0,\ \
u\big|_{t=\tau}=u_\tau,
\end{equation}
where the external forces depend explicitly on $\eb$. We assume
(following \cite{8})
 that
these external forces
are uniformly bounded in
$L^p_b(\Omega)$:
\begin{equation}\label{3.17}
\|g_\eb\|_{L^p_b(\Omega)}\le C,
\end{equation}
where $C$ is independent of $\eb$, and
converge to the limit external forces $g_0\in L^p_b(\Omega)$ in the
following weak sense: for every $\phi\in L^q(\Omega_0)$,
$\frac1q+\frac1p=1$, there exists a function
$\alpha_\phi:\R_+\to\R_+$, such that for all $h\in\R$
\begin{equation}\label{3.18}
\big|\int_{\Omega_0}(g_\eb(t+h,x)-g_0(t+h,x)).\phi(t,x)\,dx\,dt\big|
\le\alpha_\phi(\eb)\ \text{ and }\ \lim_{\eb\to0^+}\alpha_\phi(\eb)=0.
\end{equation}
The main result of this section is the following theorem.
\begin{theorem}\label{Th3.2} Let the assumptions of Theorem \ref{Th3.1} hold and
let, in addition, the external forces $g_\eb(t)$ satisfy
\eqref{3.17} and \eqref{3.18}.
 Let also $\Cal A_\eb$, $0\le\eb\le\eb_0$, be the
uniform attractors of equations \eqref{3.16}. Then, $\Cal A_\eb$
tends to $\Cal A_0$ in the following sense: for every neighborhood
$\Cal O(\Cal A_0)$
of $\Cal A_0$ in a {\rm weak} topology of $V_0^p(\omega)$ there exists
$\eb'=\eb'(\Cal O)$ such that
\begin{equation}\label{3.19}
\Cal A_\eb\subset\Cal O(\Cal A_0),\ \ \text{ if }\ \ \eb\le\eb'.
\end{equation}
\end{theorem}
\begin{proof} The proof of this theorem is based on the following lemma
which clarifies the nature of convergence \eqref{3.18}.
\begin{lemma}\label{Lem3.1} Let functions $g_\eb$, $0\le\eb\le1$, belong
to $L^p_b(\Omega)$ and satisfy \eqref{3.17} and \eqref{3.18}.
 Then, the following
conditions hold:
\par
1. The functions $g_\eb(t)$ converges to $g_0(t)$ weakly in
 $L^p_{loc}(\Omega)$ and the set $\cup_\eb\Cal H(g_\eb)$ is
weakly precompact in $L^p_{loc}(\Omega)$.
\par
2. For every sequences $\eb_n\to0$ and $\xi_n\in\Cal H(g_{\eb_n})$
such that $\xi_n\to\xi$ weakly in $L^p_{loc}(\Omega)$, the function
$\xi$ necessarily belongs to $\Cal H(g_0)$ and
\begin{equation}\label{3.20}
\big|\int_{\Omega_0}(\xi_n(t+h,x)-\xi(t+h,x)).\phi(t,x)\,dt\,dx\big|
\le\alpha_\phi(\eb_n),
\end{equation}
for every $\phi\in L^{p'}_b(\Omega_0)$ and every $h\in\R$
where $\frac1p+\frac1{p'}=1$.
\par
3. The convergence \eqref{3.18} is uniform with respect to $\phi$
belonging to compact sets in $L^{p'}(\Omega_0)$, i.e. the function
$\alpha_\phi$ in \eqref{3.18} can be chosen in such way that
\begin{equation}\label{3.21}
\alpha_V(\eb):=\sup\nolimits_{\phi\in V}\alpha_\phi(\eb)\to0,\ \ \text{ as
$\eb\to0^+$},
\end{equation}
for every compact subset $V\subset L^{p'}(\Omega_0)$.
\end{lemma}
The assertion of the lemma can be proved in a standard way, using
the representation \eqref{3.2} for functions belonging to the
halls $\Cal H(g_\eb)$ and basic properties of the weak convergence
in reflexive spaces, see \cite{8} and \cite{9}.
\par
We are now ready to prove Theorem \ref{Th3.2}. We first note that, due to
Theorem~\ref{Th1.1}, Corollary \ref{Th1.1} and estimate \eqref{3.17}, we have
\begin{equation}\label{3.22}
\|\Cal A_\eb\|_{V^p_\eb(\omega)}+
\sup_{\xi\in\Cal H(g_\eb)}\|\Cal
K^\eb_\xi\|_{W^{(1,2),p}_{\eb,b}(\Omega)}\le C,
\end{equation}
where the constant $C$ is independent of $\eb$. Thus, in order
to prove the theorem, it is sufficient to verify that, if
$u_n^0\in\Cal A_{\eb_n}$, $\eb_n\to0$ as $n\to\infty$,
 be an arbitrary sequence which converges weakly
in $V_0^p(\omega)$ to some $u_0\in V^p_0(\omega)$, then $u_0\in\Cal
A_0$ (see \cite{9}). Taking into account description \eqref{3.7},
estimates \eqref{3.17} and \eqref{3.22}
 and
the weak compactness of bounded sets in reflexive spaces, this
assertion can be reformulated as follows: if $\eb_n\to0$,
$\xi_n\in\Cal H(g_{\eb_n})$ and $u_n\in\Cal K^{\eb_n}_{\xi_n}$ be
arbitrary sequences such that $u_n\to u$ weakly in
 $W^{(1,2),p}_{0,loc}(\Omega)$ and $\xi_n\to\xi$ weakly in
$L^p_{loc}(\Omega)$, then $\xi\in\Cal H(g_0)$ and $u\in \Cal K^0_\xi$.
Let us verify this assertion. Indeed, the fact that $\xi\in \Cal H(g_0)$ is
an immediate corollary of Lemma \ref{Lem3.1}. Thus, there only remains to pass
to the weak limit in the following equations:
\begin{equation}\label{3.23}
a(\eb^2_n u_n(t)+\Dx u_n(t))-\gamma\Dt u_n(t)-f(u_n(t))=\xi_n(t),\
t\in\R,\ \ u_n\DOM=0.
\end{equation}
We recall that the embedding $W^{(1,2),p}_0(\Omega_T)\subset
C(\Omega_T)$ is compact, consequently, the weak convergence $u_n\to u$
in $W^{(1,2),p}_{loc}(\Omega)$ implies the strong convergence
 $u_n\to u$ in $C_{loc}(\Omega)$. Passing now to the limit
$n\to\infty$ in \eqref{3.23}, we derive that the function
 $u\in W^{(1,2),p}_{0,b}(\Omega)$ and satisfies
$$
a\Dx u(t)-\gamma\Dt u(t)-f(u(t))=\xi(t),\  \ t\in\R
$$
and, consequently $u\in \Cal K^0_\xi$ and Theorem \ref{Th3.1} is proved.
 \end{proof}
\begin{remark}\label{Rem3.4} Since the embedding
$V^p_0(\omega)\subset V^{p-\delta}_0(\omega)\cap C(\omega)$ is
compact, then \eqref{3.19} implies that
\begin{equation}\label{3.24}
\lim_{\eb\to0}\dist_{V^{p-\delta}_0(\omega)\cap C(\omega)}\(\Cal
A_\eb,\Cal A_0\)=0,
\end{equation}
for every $\delta>0$.
\end{remark}
To conclude this section, we consider the applications of Theorem \ref{Th3.2}
to equation \eqref{0.4} and, consequently, we assume from now on that
\begin{equation}\label{3.25}
g_\eb(t):=g(\eb^{-1}t), \ \ \text{for some } g\in L^p_b(\Omega).
\end{equation}
 \begin{example}\label{Ex3.1} Let the assumptions of Theorem \ref{Th3.1} hold,
\eqref{3.25} be satisfied and the function $g\in L^p_b(\Omega)$ have
the following heteroclinic profile structure: there exist
$g_{\pm}:=g_{\pm}(x)\in L^p(\omega)$ such that
\begin{equation}\label{3.26}
\lim_{h\to\pm\infty}\|T_hg-g_{\pm}\|_{L^p(\Omega_0)}=0.
\end{equation}
Then, obviously, $g_\eb\to g_0$ as $\eb\to0$ in $L^p_b(\Omega)$, where
\begin{equation}\label{3.27}
g_0(t):=\begin{cases} g_+, &  \text{for $t\ge0$,}\\
         g_-, &  \text{for $t<0$}
  \end{cases}
\end{equation}
and, consequently, \eqref{3.17} and \eqref{3.18} are
 also satisfied. Thus, due to
Theorem \ref{Th3.2}, the uniform attractors $\Cal A_\eb$ of equations
\eqref{3.16} (or, which is the same, the attractors of \eqref{0.4})
tend as $\eb\to0$ (in the sense of \eqref{3.19}) to the uniform
attractor of the limit parabolic equation with the external
 forces \eqref{3.27}. In particular, if $g_+=g_-$ then the limit
parabolic problem is autonomous.
\end{example}
In order to consider the next examples, we need the following
proposition which is adopted to the study  of oscillating
in time external forces $g_\eb$ in \eqref{3.16}.
\begin{proposition}\label{Prop3.1} Let $g\in L^p_b(\Omega)$ and $g_\eb$ be
defined by \eqref{3.25}. We also assume that there exists $\bar
g=\bar g(x)\in L^p(\omega)$ such that
\begin{equation}\label{3.28}
\frac1T\int_t^{t+T}g(s)\,ds\to\bar g \text{ in $L^p(\omega)$ as $T\to\infty$},
\end{equation}
 uniformly
with respect to $t\in\R$.
Then, the functions $g_\eb(t):=g(\eb^{-1}t)$, $\eb\ne0$ and
$g_0(t)\equiv \bar g$ satisfy conditions \eqref{3.17} and \eqref{3.18}.
\end{proposition}
For the proof of Proposition \ref{Prop3.1}, see  \cite{8} or \cite{9}.
\begin{example}\label{Ex3.2} Let the assumptions of Theorem \ref{Th3.1} hold,
\eqref{3.25} be satisfied and the function $g$ belong to
$C_b(\R,L^p(\omega))$ and be almost-periodic with respect to $t$
with values in $L^p(\omega)$ (the latter means that the hull $\Cal
H(g)$ is compact in $C_b(\R,L^p(\omega))$, according to the
Bochner-Amerio criterium, see \cite{18}). Then, assumption
\eqref{3.28} is satisfied, due to the Kronecker-Weyl theorem, see
\cite{18}. Thus, the uniform attractors $\Cal A_\eb$ of elliptic
problems \eqref{3.16} with the rapidly oscillating external forces
 $g_\eb(t):=g(\eb^{-1}t)$ ($g$ is now almost-periodic) converge as
 $\eb\to0$ to the global attractor $\Cal A_0$ of the limit
 autonomous parabolic equation with the averaged external forces
 $g_0\equiv \bar g$.
\end{example}
In conclusion,  we give an example of oscillating external forces
$g\in L^p_b(\Omega)$ which are not almost-periodic with respect to
time, but satisfy the assumptions of Proposition \ref{Prop3.1} (see
\cite{9} for further examples).
\begin{example}\label{Ex3.3} Let $g_1(t)$ and $g_2(t)$ be two {\it different}
$1$-periodic functions with respect to $t$ which belong to
$L^p_b(\Omega)$ and have zero mean. We set
\begin{equation}\label{3.29}
g(t):=\begin{cases} g_1(t),&;\text{for $t\in[4k^2,(2k+1)^2)$ and $k\in\Bbb Z$,}\\
       g_2(t),&;\text{for $t\in[(2k-1)^2,4k^2)$ and $k\in\Bbb Z$.}
      \end{cases}
\end{equation}
Then, obviously, this function is not almost-periodic with respect to
$t$ (even in the case where $g_1$ and $g_2$ are smooth), but condition
\eqref{3.28} is obviously satisfied
with $\bar g=0$,  since the periodic functions
$g_1$ and $g_2$ have zero mean. Thus, in this case, the attractors
$\Cal A_\eb$ of equations \eqref{3.16} with non almost-periodic rapidly
oscillating external forces $g_\eb(t):=g(\eb^{-1}t)$ converge as
$\eb\to$ to the attractor $\Cal A_0$ of the limit parabolic equation with
zero external forces.
\end{example}

\section{Local convergence as $\eb\to0$ of the individual solutions}\label{s4}

In this section, we obtain several auxiliary results on the
convergence of the solution $u_\eb(t):=U^\eb_{g_\eb}(t,\tau)u_\tau$ as
$\eb\to0$ to the corresponding solution
$u_0(t):=U^0_{g_0}(t,\tau)u_\tau$ of the limit parabolic problem
which will be essentially used in the next sections. We also assume
(for simplicity) that condition \eqref{3.18} is satisfied with
the {\it autonomous} limit function $g_0\equiv \bar g\in L^p(\omega)$.
Then, equations \eqref{3.16} converge as $\eb\to0$ to the following
autonomous reaction-diffusion problem:
\begin{equation}\label{4.1}
\gamma\Dt u_0=a\Dx u_0-f(u_0)+\bar g,\ \ u_0\DOM=0,\ \
u_0\big|_{t=\tau},
\end{equation}
which generates a dissipative semigroup $S_t:=U^0_{\bar g}(t,0)$
in the phase space $V^p_0(\omega)$ and possesses the global attractor
 $\Cal A_0\subset V^p_0(\omega)$, see Theorems \ref{Th1.1}, \ref{Th2.1} and \ref{Th3.1}. The
following theorem gives the estimate for the  $L^2(\omega)$-norm of
 distance between
$U^\eb_{g_\eb}(t,\tau)$ and $S_{t-\tau}$.
\begin{theorem}\label{Th4.1} Let the assumptions of Theorem \ref{Th3.2} hold,
$p>2p_{min}$ and $g_0(t)\equiv \bar g\in L^p(\omega)$. Then, for every
$R>0$, there exist a function $\alpha_R:\R_+\to\R_+$,
$\lim_{\eb\to0}\alpha_R(\eb)=0$ and a positive constant $K$
 such that, for every $\eb\le\eb'_0$,
$h_\eb\in\Cal H(g_\eb)$, $\tau\in\R$, $t\ge\tau$ and $u_\tau\in V^p_\eb(\omega)$
satisfying $\|u_\tau\|_{V^p_0(\omega)}\le R$, the following
estimate holds:
\begin{equation}\label{4.2}
\|U^\eb_{h_\eb}(t,\tau)u_\tau-S_{t-\tau}u_\tau\|_{L^2(\omega)}\le
\alpha_R(\eb)e^{K(t-\tau)}.
\end{equation}
\end{theorem}
\begin{proof} We set $u_\eb(t):=U^\eb_{h_\eb}(t,\tau)u_\tau$,
$u_0(t):=S_{t-\tau}u_\tau$ and $v_\eb(t):=u_\eb(t)-u_0(t)$. Then, the
last function satisfies
\begin{equation}\label{4.3}
\gamma\Dt v_\eb-a\Dx v_\eb+l_\eb(t)v_\eb=
a\eb^2\Dt^2u_\eb(t)+(h_\eb(t)-\bar g),\ \ v_\eb\DOM=0,\
v_\eb\big|_{t=\tau}=0,
\end{equation}
where $l_\eb(t):=\int^1_0f'(su_\eb(t)+(1-s)u_0(t))\,dt$. Multiplying
equation \eqref{4.3} by $v_\eb(t)$, integrating over $(\tau,t)\times\Omega$
and using that $l_\eb(t)\ge-K$ and $\gamma=\gamma^*>0$, we have
\begin{multline}\label{4.4}
\|v(t)\|_{L^2(\omega)}^2-K_1\int_\tau^t\|v(s)\|_{L^2(\omega)}^2\,ds\le\\
\le C\eb^2\int^t_\tau\(\|\Dt u_\eb(s)\|_{L^2(\omega)}^2+
\|\Dt u_0(s)\|_{L^2(\omega)}^2\)\,ds+C\eb^4\|\Dt
u_\eb(t)\|_{L^2(\omega)}^2+\\+
C\big|\int^t_\tau\int_{\omega}(h_\eb(s)-\bar g).v(s)\,ds\,dx\big|,
\end{multline}
where the constants $C$ and $K_1$ depend only on $\gamma$, $K$ and $a$. We now
note that, due to the assumption $p>2p_{min}$, we may apply estimate
\eqref{1.6} with the exponent $p/2>2$ instead of $p$ to equations
\eqref{3.16}. Then, using Lemma \ref{LemA.1}, estimate \eqref{3.17}
 and embedding \eqref{2.28}, we
have
\begin{equation}\label{4.5}
\|u_\eb\|_{W^{1,2}(\Omega_T)}^2+\|u_0\|_{W^{1,2}(\Omega_T)}^2+
\eb^3\|\Dt u_\eb(T)\|_{L^2(\omega)}^2+\|u_\eb\|_{L^\infty(\Omega_T)}^2+
\|u_0\|_{L^\infty(\Omega_T)}^2
\le C_R,
\end{equation}
where the constant $C_R$ is independent on
$\|u_\tau\|_{V_0^p(\omega)}$ (which satisfies
$\|u_\tau\|_{V^p_0(\omega)}\le R$), $\eb\le\eb_0$, $h_\eb\in\Cal
H(g_\eb)$, $\tau\in\R$ and $t\ge\tau$. Thus, there only remains to
estimate the last term in \eqref{4.4}. To this end, we recall that
the limit function $g_0(t)\equiv \bar g$ in \eqref{3.18} is now
independent of $t$. Consequently, Lemma \ref{Lem3.1} implies that every
sequence $h_{\eb_k}\in\Cal H(g_{\eb_k})$, $\eb_k\to0$ converges weakly
in $L^p_{loc}(\Omega)$ to the mean value $\bar g$ and (due to
\eqref{3.20}), for every $\phi\in L^{p'}(\Omega_0)$, there exists
$\alpha_\phi(\eb)\to0$ as $\eb\to0$ such that
\begin{equation}\label{4.6}
\big|\int_0^1 \int_\omega(h_\eb(t+s,x)-\bar g(x)).\phi(t,x)\,dx\,dt\big|
\le\alpha_\phi(\eb),\ \ \forall s\in\R \text{ and } h_\eb\in\Cal H(g_\eb).
\end{equation}
Moreover, this convergence is uniform with respect to $\phi$ belonging
to compact sets in $L^{p'}(\Omega_0)$. We now note that, due to
\eqref{4.5}, the set of functions $\{T_sv_\eb,\,s\ge\tau, \,
\|u_\tau\|_{V^p_0(\omega)}\le R,\ h_\eb\in\Cal H(g_\eb),\eb\le\eb_0'\}$
is bounded in $W^{1,2}(\Omega_0)$ and, consequently, it is compact
in $L^{p'}(\Omega_0)$
(we recall that $\frac1p+\frac1{p'}=1$ and $p>2$, consequently, $p'<2$).
 Therefore, \eqref{4.6} implies that
\begin{equation}\label{4.7}
\big|\int_{\tau+s}^{\tau+s+1}
\int_\omega(h_\eb(t,x)-\bar g).v_\eb(t,x)\,dx\,dt\big| \le\tilde\alpha_R(\eb),
\end{equation}
where $\tilde\alpha_R(\eb)\to0$ as $\eb\to0$ uniformly with respect to
$h_\eb\in\Cal H(g_\eb)$, $s\in\R_+$, $\eb\le\eb_0'$ and
 $\|u_\tau\|_{V^p_0(\omega)}\le R$. Inserting estimates \eqref{4.5}
and \eqref{4.7} to \eqref{4.4}, we have
\begin{equation}\label{4.8}
\|v(t)\|_{L^2(\omega)}^2-K_1\int_\tau^t\|v(s)\|_{L^2(\omega)}^2\,ds\le
C_R(t-\tau)\eb+C(t-\tau)\tilde\alpha_R(\eb).
\end{equation}
Applying the Gronwall inequality to estimate \eqref{4.8}, we finish
the proof of Theorem~\ref{Th4.1}.
\end{proof}
The following corollary reformulates estimate \eqref{4.2} in terms
of discrete cascades \eqref{2.27} acting on the phase space
$V^p_0(\omega)$.
\begin{corollary}\label{Cor4.1} Let the assumptions of Theorem \ref{Th4.1} hold and
let, in addition, the external forces $g_\eb$ be uniformly bounded
in $L^{p+\delta}_b(\Omega)$, for some $\delta>0$, i.e.
\begin{equation}\label{4.9}
\|g_\eb\|_{L^{p+\delta}_b(\Omega)}\le C,
\end{equation}
where $C$ is independent of $\eb$. Then, the following estimate is
valid:
\begin{equation}\label{4.10}
\|U^\eb_{h_\eb}(l,m)u_m-S_{l-m}u_m\|_{V^p_0(\omega)}\le
e^{K'(l-m)}\bar\alpha_R(\eb),\ \ \|v_m\|_{V^p_0(\omega)}\le R,
\end{equation}
where the constant $K'$ and the function $\bar\alpha_R$
($\bar\alpha_R(\eb)\to0$ as $\eb\to 0^+$) are independent of
$\eb\le\eb_0$, $h_\eb\in\Cal H(g_\eb)$, $u_m\in V^p_0(\omega)$ and $l,m\in\Bbb Z$
(with $l\ge m$).
\end{corollary}
\begin{proof} Since the functions $g_\eb$ are assumed to be uniformly
bounded in $L^{p+\delta}_b(\omega)$, then (replacing the exponent $p$
by $p+\delta$) we derive from Theorem \ref{Th1.1} and Corollary \ref{Cor1.2}
(analogously to \eqref{2.29}) that
\begin{equation}\label{4.11}
\|U^\eb_{h_\eb}(l,m)u_m-S_{l-m}u_m\|_{V^{p+\delta}_0(\omega)}\le C_R,
\end{equation}
where the constant $C_R$ is independent of $\eb$, $h_\eb$, $l,m$ and
$u_m$. Estimate \eqref{4.10} is an immediate corollary of
\eqref{4.2}, \eqref{4.11} and the following interpolation inequality:
\begin{equation}\label{4.12}
\|w\|_{V^p_0(\omega)}\le C\|w\|_{L^2(\omega)}^{\kappa_\delta}
\cdot\|w\|_{V^{p+\delta}_0(\omega)}^{1-\kappa_\delta},
\end{equation}
for the appropriate $0<\kappa_\delta<1$ (see, e.g. \cite{25}) and
Corollary \ref{Cor4.1} is proved.
\end{proof}
Our next task is to obtain the analogue of Theorem \ref{Th4.1} and Corollary
\ref{Cor4.1} for the Frechet derivatives of the processes
$U^\eb_{g_\eb}(t,\tau)$.
\begin{theorem}\label{Th4.2} Let the assumptions of Theorem \ref{Th4.1} hold.
Then, for every $R\in\R_+$ and $u_\tau$,
$\|u_\tau\|_{V^p_0(\omega)}\le R$,
the following estimate is valid:
\begin{equation}\label{4.13}
\|D_u U^\eb_{h_\eb}(t,\tau)(u_\tau)-D_u S_{t-\tau}(u_\tau)\|_{\Cal
L(V^p_0(\omega), L^2(\omega))}\le  e^{K''(t-\tau)}\alpha_R(\eb),
\end{equation}
where the constant $K''$ and the function $\alpha_R$
($\alpha_R(\eb)\to0$ as $\eb\to 0^+$) are independent of
$\eb\le\eb_0$, $h_\eb\in\Cal H(g_\eb)$, $u_\tau\in V^p_0$ and $t,\tau\in\Bbb R$
(with $t\ge \tau$).
\end{theorem}
\begin{proof} We set $w_\eb(t):=D_uU^\eb_{h_\eb}(t,\tau)(u_\tau)\xi$ and
$w_0(t):=D_u S_{t-\tau}(u_\tau)\xi$, where $\xi\in V^p_0(\omega)$ is
an arbitrary vector. Then, according to Theorem \ref{Th2.2},
 these functions satisfy the equations
\begin{equation}\label{4.14}
\begin{cases}
a(\eb^2\Dt^2 w_\eb+\Dx w_\eb)-\gamma\Dt w_\eb-f'(u_\eb(t))w_\eb=0,\
w_\eb\DOM=0,\  w_\eb\big|_{t=\tau}=\xi,\\
a\Dx w_0-\gamma\Dt w_0-f'(u_0(t))w_0=0,\ \
w_0\DOM=0,\ \  w_0\big|_{t=\tau}=\xi,
\end{cases}
\end{equation}
where $u_0(t)$ and $u_\eb(t)$ are the same as in the proof of Theorem
\ref{Th4.1}. Then, according to Lemma \ref{Lem2.1}, embeddings \eqref{2.28} and
Corollary \ref{CorA.2}, analogously to \eqref{4.5}, we have
\begin{multline}\label{4.15}
\|w_\eb\|_{W^{1,2}(\Omega_T)}^2+\eb^3\|\Dt w_\eb(T)\|_{L^2(\omega)}^2+
\|w_0\|_{W^{1,2}(\Omega_T)}^2+\\+
\|w_0\|_{L^\infty(\Omega_T)}^2+\|w_\eb\|_{L^\infty(\Omega_T)}^2\le
C_R e^{2\Lambda_0(T-\tau)}\|\xi\|_{V^p_0(\omega)}^2,
\end{multline}
where the constant $C_R$ is independent of $\eb\le\eb_0'$, $h_\eb\in\Cal
H(g_\eb)$, $\xi\in V^p_0(\omega)$, $u_\tau\in V^p_0(\omega)$
(which satisfies $\|u_\tau\|_{V^p_0(\omega)}\le R$), $\tau\in\R$
and $T\ge\tau$.
\par
We now set $\theta_\eb(t):=w_\eb(t)-w_0(t)$. Then, this function
satisfies
$$
\gamma\Dt\theta_\eb-a\Dx\theta_\eb+f'(u_\eb(t))\theta_\eb=
a\eb^2\Dt^2w_\eb(t)-[f(u_\eb(t))-f(u_0(t))]w_\eb(t),\ \
\theta_\eb\big|_{t=\tau}=0.
$$
Multiplying this equation by $\theta_\eb(t)$, integrating over
 $(T,\tau)\times\omega$ and using that $f'\ge-K$ and the functions
$u_\eb(t)$ and $u_0(t)$ are uniformly bounded in the $L^\infty$-norm,
we have (analogously to \eqref{4.4})
\begin{multline}\label{4.16}
\|\theta_\eb(t)\|_{L^2(\omega)}^2-
K_3\int^T_\tau\|\theta_\eb(t)\|_{L^2(\omega)}^2\,dt\le\\\le
C\eb^2\int^T_\tau\(\|\Dt w_\eb(t)\|_{L^2(\omega)}^2+
\|\Dt w_0(t)\|_{L^2(\omega)}^2\)\,dt+C\eb^4\|\Dt
w_\eb(T)\|_{L^2(\omega)}^2+\\+
C_R\int^T_\tau\|u_\eb(t)-u_0(t)\|_{L^2(\omega)}^2
\|w_\eb(t)\|_{L^\infty(\omega)}^2\,dt,
\end{multline}
where the constants $C$, $K_3$ and $C_R$ are independent of $\eb$ and
$T$. Inserting estimates \eqref{4.2} and \eqref{4.15} to the
right-hand side of \eqref{4.16}, we have
\begin{equation}\label{4.17}
\|\theta_\eb(T)\|_{L^2(\omega)}^2-K_3
\int^T_\tau\|\theta_\eb(t)\|_{L^2(\omega)}^2\,dt\le
C'_R(T-\tau)(\eb+\alpha_R(\eb)^2)
e^{4\Lambda_0(T-\tau)}\|\xi\|_{V^p_0(\omega)}^2.
\end{equation}
Applying the Gronwall inequality to this estimate, we derive estimate
\eqref{4.13} (with the appropriate new function $\alpha_R$) and
finish the proof of Theorem \ref{Th4.2}.
\end{proof}
The following corollary is the analogue of Corollary \ref{Cor4.1} for the
Frechet derivatives.
\begin{corollary}\label{Cor4.2} Let the assumptions of Theorem \ref{Th4.1} hold and
let, in addition, the external forces $g_\eb$ be uniformly bounded
in $L^{p+\delta}_b(\Omega)$, for some $\delta>0$ (i.e., \eqref{4.9}
be satisfied)
Then, the following estimate is
valid, for every $R\in\R_+$ and $u_m\in V_0^p(\omega)$ with
$\|u_m\|_{V^p_0(\omega)}\le R$:
\begin{equation}\label{4.18}
\|D_uU^\eb_{h_\eb}(l,m)(u_m)-D_uS_{l-m}(u_m)
\|_{\Cal L(V^p_0(\omega),V_0^p(\omega))}\le
e^{K''(l-m)}\bar\alpha_R(\eb),
\end{equation}
where the constant $K''$ and the function $\bar\alpha_R$
($\bar\alpha_R(\eb)\to0$ as $\eb\to 0^+$) are independent of
$\eb\le\eb_0$, $h_\eb\in\Cal H(g_\eb)$, $u_m\in V^p_0$ and $l,m\in\Bbb Z$
(with $l\ge m$).
\end{corollary}
Indeed, it follows from \eqref{4.15} and Corollary \ref{CorA.2}
(where the exponent $p$ is replaced by $p+\delta$)
 that
\begin{equation}\label{4.19}
\|D_uU^\eb_{h_\eb}(l,m)(u_m)-D_uS_{l-m}(u_m)
\|_{\Cal L(V^p_0(\omega),V_0^{p+\delta}(\omega))}\le
C_R'e^{\Lambda_0(l-m)},
\end{equation}
 where $l\ge m+1$ and  the constant $C'_R$ is independent of $\eb$, $l$ and
$m$. Estimate \eqref{4.18} is an immediate corollary of
\eqref{4.19}, \eqref{4.13} and \eqref{4.12}.
\par
To conclude this section, we investigate the dependence of functions
$\alpha_R(\eb)$ and $\bar\alpha_R(\eb)$ which are introduced in
Theorems \ref{Th4.1}-\ref{Th4.2} and Corollaries \ref{Cor4.1}-\ref{Cor4.2} respectively on $\eb\to0$.
For simplicity, we assume that the external forces $g_\eb(t)$
satisfy \eqref{3.25} (i.e., $g_\eb(t):=g(\eb^{-1}t)$) and
and $g\in L^{p+\delta}_b(\Omega)$.
\begin{theorem}\label{Th4.3} Let the assumptions of Theorem \ref{Th4.1} hold and
\eqref{3.25} be satisfied for some $g\in L^{p+\delta}_b(\Omega)$,
$\delta>0$. Assume also that the function $g(t)-\bar g$ has a bounded
primitive, i.e.
\begin{equation}\label{4.20}
g(t)-\bar g=\partial_t G(t),
\end{equation}
for some $G\in L^2_b(\omega)$. Then, the functions $\alpha_R(\eb)$
 in Theorems \ref{Th4.1} and \ref{Th4.2} and the functions $\bar\alpha_R(\eb)$
in Corollaries \ref{Cor4.1} and \ref{Cor4.2} have the following structure:
\begin{equation}\label{4.21}
\alpha_R(\eb):=C_R\eb^{1/2},\ \
\bar\alpha_R(\eb):=\bar C_R\eb^{\kappa_\delta/2},
\end{equation}
where $\kappa_\delta$ is the same as in \eqref{4.12} and the
constants $C_R$ and $\bar C_R$ depend on $R$, but are independent of
$\eb$.
\end{theorem}
\begin{proof} We first note that, due to representation \eqref{3.2}
and the weak compactness of bounded subsets of $L^p_{loc}(\Omega)$,
it follows from \eqref{4.20} that, for every $h\in\Cal H(g)$,
there exists $H\in\Cal H(G)$ such that
\begin{equation}\label{4.22}
h(t)-\bar g=\Dt H(t).
\end{equation}
We also note that, according to  \eqref{4.4}-\eqref{4.5} and \eqref{4.17},
it is sufficient to verify that
\begin{equation}\label{4.23}
I_\eb(T):=\big|\int^T_\tau\int_\omega(h_\eb(t,x)-\bar
g(x)).v_\eb(t,x)\,dt\,dx\big|\le C_R''(T-\tau+1)\eb,
\end{equation}
where $h\in\Cal H(g)$, $h_\eb(t):=h(\eb^{-1}t)$ and the constant $C_R$
is independent of $\eb$. Let us verify this inequality. Indeed, it
follows from \eqref{4.22} that
\begin{equation}\label{4.24}
h_\eb(t)-\bar g=\eb\Dt H_\eb(t),\ \ H_\eb(t):=H(\eb^{-1}t),
\end{equation}
consequently, integrating by parts in \eqref{4.23}, we have
\begin{multline}\label{4.25}
I_{\eb}(T)\le \eb\big|\int_\tau^T\int_\omega
H_\eb(t).\Dt v_\eb(t)\,dx\,dt\big|+\eb\big|\int_\omega
H_\eb(T).v_\eb(T)\,dx\big|\le\\\le C\eb(T-\tau+1)\(\|G\|_{L^2_b(\Omega)}+
\|\Dt G\|_{L^2_b(\Omega)}\)\|v_\eb\|_{W^{1,2}_b(\Omega)},
\end{multline}
where the constant $C$ is independent of $\eb$ (here we implicitly
used that $\|H_\eb\|_{L^2_b(\Omega)}\le\|H\|_{L^2_b(\Omega)}
\le \|G\|_{L^2_b(\Omega)}$). Estimate \eqref{4.23} is an immediate
corollary of \eqref{4.25} and \eqref{4.5}. Theorem \ref{Th4.3} is proved.
\end{proof}


\section{Nonautonomous unstable manifolds of the nonlinear
elliptic equation}\label{s5}
In this section, using the standard perturbation
technique,
we define, for a sufficiently small $\eb>0$, the nonautonomous unstable manifold
of nonautonomous elliptic system \eqref{1.1} which corresponds to
the hyperbolic equilibrium of the limit parabolic equation \eqref{4.1}. This result
will be used in the next section in order to construct the nonautonomous
regular attractor for system \eqref{1.1}.
For simplicity, we restrict ourselves to consider
only the case of rapidly oscillating external forces
$g_\eb(t):=g(\eb^{-1}t)$ where $g$ is an almost-periodic function
with respect to $t$ with values in $L^{p+\delta}(\omega)$:
\begin{equation}\label{5.1}
g\in AP(\R,L^{p+\delta}(\omega)),\ \ p>2p_{min},\ \ \delta>0,
\end{equation}
although all of the results formulated below remain true (after minor
changing) under assumptions of previous section.
Moreover, since the results of this section are more or less standard,
we give below only schematic proofs resting the details to the reader.
\par
Our main assumption is
that equation \eqref{4.1} possesses a hyperbolic equilibrium
$z_0\in V^p_0(\omega)$, i.e.
\begin{equation}\label{5.2}
a\Dx z_0-f(z_0)+\bar g=0\ \ \text{and}\ \ \sigma(L_{z_0})\cap
\{\Ree \lambda=0\}=\varnothing
\end{equation}
where $L_{z_0}:=\gamma^{-1}(a\Dx-f'(z_0))$ and
$\sigma(L)$ denotes the spectrum of the operator $L$. It is well known
that hyperbolicity assumption \eqref{5.2} implies existence of
the spectral decomposition
\begin{equation}\label{5.3}
V^p_0(\omega)=V_++V_-,\ \ V_+\cap V_-=\{0\},
\end{equation}
where the linear subspaces $V_\pm$ are invariant with respect to the operator
$L_{z_0}$ and satisfy the following properties:
$$
\sigma(L_{z_0}\big|_{V_+})\subset\{\Ree\lambda\ge\nu\},\ \  \
\sigma(L_{z_0}\big|_{V_-})\subset\{\Ree\lambda\le-\nu\},
$$
for a sufficiently small positive $\nu$. Moreover, the dimension of the
unstable subspace $V_+$ is finite and is called the instability index of
the equilibrium $z_0$:
\begin{equation}\label{5.4}
\ind_{z_0}:=\dim V_+<\infty,
\end{equation}
see e.g. \cite{4} and \cite{25}. It is also well known (see
\cite{4}, \cite{24}) that spectral decomposition \eqref{5.3}
generates two invariant manifolds $\Cal M^+_{z_0}$ and $\Cal
M^-_{z_0}$ for {\it nonlinear} problem \eqref{4.1} in a
sufficiently small neighborhood of $z_0$ which correspond to linear
subspaces $V_+$ and $V_-$ respectively. Since only unstable
manifolds are important for the attractors theory, we formulate
below the rigorous result on the unstable manifold $\Cal M^+_{z_0}$
only.
\begin{theorem}\label{Th5.1} Let the assumptions of Theorem \ref{Th4.1}
hold and let, in addition, \eqref{5.2} be satisfied. Then, there
exists a small neighborhood $\Cal W_{z_0}$ of the equilibrium $z_0$
(in the $V^p_0$-metric) such that the set
\begin{multline}\label{5.5}
\Cal M^{+,loc}_{z_0}:=\big\{u_0\in V^p_0(\omega)\,,\ \exists u\in C_b(\R,V^p_0(\omega))
\text{ which solves \eqref{4.1}, $u(0)=u_0$, }\\
\lim\nolimits_{t\to-\infty}u(t)=z_0
\ \text{ and }\ u(t)\in \Cal W_{z_0}\ \ \forall t\le0\big\}
\end{multline}
is a finite dimensional $C^1$-submanifold of $V^p_0(\omega)$ which is
diffeomorphed to $V_+$.
\end{theorem}
The main task of this Section is to construct the analogue of unstable
manifold \eqref{5.5} for nonlinear {\it elliptic} equation \eqref{3.16}
if $\eb>0$ is small enough. Since equation \eqref{3.16} is 'nonautonomous'
then we first need to find the analogue of the hyperbolic equilibrium $z_0$.
\begin{proposition}\label{Prop5.1} Let the assumptions of Theorem \ref{Th4.1} hold
and \eqref{5.1}-\eqref{5.2} be satisfied. Then there exist $\tilde\eb_0>0$
and $\delta_0>0$
such that, for every $\eb\le\tilde\eb_0$, $h\in\Cal H(g)$
there exists a unique solution
$u=u_{h_\eb,z_0}^\eb(t)$ of the problem
\begin{equation}\label{5.6}
a(\eb^2\Dt^2 u+\Dx u)-\gamma\Dt u-f(u)=h_\eb(t),\ t\in\R,\
\ h_\eb(t):=h(\eb^{-1}t),
\end{equation}
which satisfies the condition
$\|u_{h_\eb,z_0}^\eb-z_0\|_{C_b(\R,V^p_0(\omega))}\le\delta_0$.
Moreover, this solution is almost-periodic with respect to $t$:
\begin{equation}\label{5.7}
u_{h_\eb,z_0}^\eb\in AP(\R,V^p_\eb(\omega)) \text{ and }\ \
\|u_{h_\eb,z_0}^\eb-z_0\|_{C_b(\R,V^p_0(\omega))}\le C\bar\alpha_{R_0}(\eb),
\end{equation}
where the constants $R_0$ and $C$ are independent of $\eb$ and $h$
and the function $\bar \alpha_{R_0}(\eb)$ is the same as in
Corollaries \ref{Cor4.1} and \ref{Cor4.2}. In particular, $u_{h_\eb,z_0}^\eb\to z_0$ as
$\eb\to0$.
\end{proposition}
{\it Sketch of proof.} Instead of solving \eqref{5.6}, we first solve the
following discrete analogue of this equation:
\begin{equation}\label{5.8}
u(n+1)=U^\eb_{h_\eb}(n+1,n)u(n)
\end{equation}
in the space of sequences $u\in l^\infty(\Bbb Z,
V^p_0(\omega))$. We are going to solve \eqref{5.8} near the constant
sequence $u(n)\equiv z_0$ using the implicit function
theorem. Indeed, due Corollaries \ref{Cor2.4}, \ref{Cor4.1} and \ref{Cor4.2}, the operators
$U^\eb_{h_\eb}(n+1,n)$ are close to $S_1$ together with their Frechet
derivatives as $\eb\to0$ (uniformly with respect to $n$ and
$h$). Moreover, the linearized problem (which corresponds to
 \eqref{5.8} at $\eb=0$ and $u=z_0$)
\begin{equation}\label{5.9}
w(n)-D_uS_1(z_i)w(n)=\tilde h(n)
\end{equation}
is uniquely solvable in $l^\infty(\Bbb Z,V^p_0(\omega))$ for every
$\tilde h\in l^\infty(\Bbb Z,V^p_0(\omega))$ (due to hyperbolicity
of the equilibrium $z_i$. Thus, the implicit function theorem is
indeed applicable to equation \eqref{5.8} and gives the existence
and uniqueness of the solution $\bar u_{h_\eb,z_0}^\eb\in
l^\infty(\Bbb Z, V^p_0(\omega))$ of \eqref{5.8}, for sufficiently
small $\eb>0$, which belongs to a small neighborhood of the
equilibrium $z_0$. Moreover, it follows in a standard way from
\eqref{2.30}, \eqref{4.10} and \eqref{4.18} that this solution
satisfies
\begin{equation}\label{5.10}
\|\bar u_{h_\eb,z_0}^\eb(n)-z_0\|_{V^p_0(\omega)}\le C\bar\alpha_{R_0}(\eb),\
\forall n\in\Bbb Z,
\end{equation}
where $C$ is independent of $h$, $i$ and $\eb$ and $R_0$ is the radius
of the uniform (with respect to $\eb$ and $h$) absorbing ball in
 $V^p_0(\omega)$ for the discrete processes $U^\eb_{h_\eb}(l,m)$,
$l,m\in\Bbb Z$, $l\ge m$, and $h\in\Cal H(g)$ (which exists due to
estimate \eqref{2.29}). The desired continuous function
$u_{h_\eb,z_0}^\eb(t)$, $t\in\R$, can be now defined as follows:
\begin{equation}\label{5.11}
u_{h_\eb,z_0}^\eb(t):=u_{T_th_\eb,z_0}^\eb(0).
\end{equation}
Obviously, \eqref{5.11} is a solution of \eqref{5.6} which
belongs to the space $C_b(\R,V^{p+\delta}_\eb(\omega))$ (due to
Corollary \ref{Cor1.2}) and satisfies
\begin{equation}\label{5.12}
\|u^\eb_{h_\eb,z_0}\|_{C_b(\R,V^{p+\delta}_\eb(\omega))}\le C,
\end{equation}
where the constant $C$ is independent of $\eb$, $i$ and $h$.
The uniqueness of this solutions (in a small neighborhood of $z_0$) is
an immediate corollary of the uniqueness of the discrete solution
$\bar u^\eb_{h_\eb,z_0}(m)$.
 The
almost-periodicity of this function is a standard corollary of that
uniqueness, see  e.g. \cite{18}). Proposition \ref{Prop5.1} is proved.

\qed

Now we are ready to define the analogues of the unstable sets $\Cal
M^+_{z_0}$ for problem \eqref{5.6}. Since this problem is
'nonautonomous' then these manifolds also depend on~$t$.
\begin{definition}\label{Def5.1} Let the assumptions of Proposition \ref{Prop5.1}
hold. For every $\eb\le\hat\eb_0$, $h\in\Cal H(g)$
 and $t\in\R$, we
define the set $\Cal M^{+,loc}_{\eb,h_\eb,z_0}(\tau)$ as follows:
\begin{multline}\label{5.13}
\Cal M^{+,loc}_{\eb,h_\eb,z_0}(\tau):=\big\{u_\tau\in V^p_\eb(\omega),\ \
\exists u\in C_b(\R,V^p_\eb(\omega)), \text{ which solves \eqref{5.6}},\\
\text{ $u(\tau)=u_\tau$, }\
\lim_{t\to-\infty}\|u(t)-u^\eb_{h_\eb,z_0}(t)\|_{V^p_\eb(\omega)}=0,\ \
u(t)\in\Cal W_{z_0}\ \ \forall t\le\tau\big\},
\end{multline}
where $u^\eb_{h_\eb,z_0}(t)$ is the solution of \eqref{5.6} constructed in
Proposition \ref{Prop5.1} and $\Cal W_{z_0}$ is the same as in Theorem \ref{Th5.1}.
Thus, set \eqref{5.13} consists of the values $u(\tau)$
at moment $\tau$ of all solutions $u\in C_b(\R,V^p_\eb(\omega))$
 of \eqref{5.6} which tend to
$u^\eb_{h_\eb,z_0}(t)$ as $t\to-\infty$ and belong to $\Cal W_{z_0}$
 for $t\le\tau$.
\end{definition}
The following theorem is the 'nonautonomous' analogue of Theorem \ref{Th5.1}.
\begin{theorem}\label{Th5.2} Let the assumptions of Proposition \ref{Prop5.1} hold.
Then, for sufficiently small $\eb\le\eb_0\ll1$ and
every $\tau\in\R$, the sets \eqref{5.18}
are $C^1$-submanifolds of $V^p_\eb(\omega)$ which are diffeomorphed to $V_+$.
\end{theorem}
{\it Sketch of the proof.} We first note that
\begin{equation}\label{5.14}
\Cal M^{+,loc}_{\eb,h_\eb,z_0}(\tau)=\Cal M^{+,loc}_{\eb, T_\tau h_\eb,z_0}(0)
\end{equation}
and, consequently, it is sufficient to prove the theorem for $\tau=0$ only.
In order to do so, we seek for the solution $u(n)$, $n\le0$, of the
following discrete problem
\begin{equation}\label{5.15}
u(n+1)=U^\eb_{h_\eb}(n+1,n)u(n),\ \ n\in\Bbb Z_-, \ \ u(n)\in\Cal W_{z_0}
\end{equation}
in the space $\Theta_-:=l^\infty(\Bbb Z_-,V^p_0(\omega))$ of
$V^p_0$-valued sequences which remain bounded as $n\to-\infty$.
Then, as known (see e.g. \cite{12} and \cite{13,29})
 the union of all initial values $u(0)$ for such sequences gives the
desired unstable set $\Cal M^{+,loc}_{\eb,h_\eb,z_0}(0)$. Thus, there remains
to prove that the set of all solutions $u(n)$ of \eqref{5.15} belonging
to $\Theta_-$ is a submanifold of $\Theta_-$. To this end, we introduce
a new sequence  $w(n):=u(n)-u^\eb_{h_\eb,z_0}(n)$ where $u^\eb_{h_\eb,z_0}(t)$ is
defined in Proposition \ref{Prop5.1} and consider the following problem:
\begin{multline}\label{5.16}
w(n+1)=\Bbb T_{n,h_\eb,\eb}w(n),\ \ \Pi_+w(0)=w_0\in V_+,\\
\Bbb T_{n,h_\eb,\eb}w:=U^\eb_{h_\eb}(n+1,n)(u^\eb_{h_\eb,z_0}(n)+w)-
U^\eb_{h_\eb}(n+1,n)u^\eb_{h_\eb,z_0}(n),
\end{multline}
where $\Pi_+$ is a spectral projector to the spectral space $V_+$.
For every $w_0$, we find the solution $w\in\Theta_-$ using the implicit
function theorem. Indeed, due to Corollaries \ref{Cor2.4}, \ref{Cor4.1} and \ref{Cor4.2} and
Proposition \ref{Prop5.1}, the operators $\Bbb T_{n,h_\eb,\eb}(w)$ tend together
with their Frechet derivative to $S_1(w+z_0)-z_0$
(where $S_t$ is a solving semigroup of the limit parabolic equation
\eqref{4.1}) as $\eb\to0$ uniformly with respect to $h\in\Cal H(g)$
and $n\in\Bbb Z$. Moreover, the linearized (at $w=0$) limit equation
\begin{equation}\label{5.17}
v(n+1)=D_{u_0}S_1(z_0)v(n),\ \ \Pi_+v(0)=v_0,\ \ n\in\Bbb Z_-
\end{equation}
is uniquely solvable in $\Theta_-$, for every $v_0\in V_+$ (due to
the hyperbolicity of the equilibrium $z_0$). Applying the implicit
function theorem to equation \eqref{5.16}, we derive that, for
every $\eb\le\eb_0$ and every $w_0$ belonging to a sufficiently
small neighborhood of zero in $V_+$, there exists a unique solution
$w\in\Theta_-$ of \eqref{5.16} and that the set of all such
solutions is a $C^1$-submanifold in $\Theta_-$ diffeomorphed to
$V_+$. Thus, we have verified that $\Cal
M^{+,loc}_{\eb,h_\eb,z_0}(0)$ is a submanifold in $V^p_0(\omega)$.
Using now the smoothing property for the operators
$U^\eb_{h_\eb}(t,\tau)$ and embeddings \eqref{2.28}, it is easy to
show that $\Cal M^{+,loc}_{\eb,h_\eb,z_0}(0)$ is a submanifold not
only in $V^p_0(\omega)$, but also in $V^p_\eb(\omega)$. Theorem \ref{Th5.2}
is proved.

\qed

\begin{remark}\label{Rem5.1} It is not difficult to verify that the
manifolds $\Cal M^{+,loc}_{\eb,h_\eb,z_0}(\tau)$ are almost
periodic with respect to $\tau$  and tend to the unstable manifold
$\Cal M^{+,loc}_{z_0}$ of the limit parabolic problem \eqref{4.1}
as $\eb\to0$ uniformly with respect to $h\in\Cal H(g)$ and
$\tau\in\R$, \cite{12,13} and \cite{15} for the details.
\end{remark}
Finally, we define the {\it global} unstable manifolds
$\Cal M^+_{\eb,h_\eb,z_0}(\tau):=\Cal M^{+,gl}_{\eb,h_\eb,z_0}(\tau)$ as
a union of all values $u(\tau)$ for all solutions $u\in C_b(\R, V^p_\eb(\omega))$
of problem \eqref{5.6} which tend to $u^\eb_{h_\eb,z_0}(t)$ as $t\to-\infty$.
Then, these sets can be expressed in terms of the local
unstable manifolds via
\begin{equation}\label{5.18}
\Cal M^+_{\eb,h_\eb,z_0}(\tau)=
\cup_{n=1}^\infty U^\eb_{h_\eb}(\tau,\tau-n)
\Cal M^{+,loc}_{\eb,h_\eb,z_0}(\tau-n).
\end{equation}
Moreover, these sets are, obviously, strictly invariant with respect to
the dynamical process $U^\eb_{h_\eb}(t,\tau)$:
\begin{equation}\label{5.19}
U^\eb_{h_\eb}(t,\tau)\Cal M^+_{\eb,h_\eb,z_0}(\tau)=\Cal M^+_{\eb,h_\eb,z_0}(t)
\end{equation}
and satisfy the following translation property:
\begin{equation}\label{5.20}
\Cal M^+_{\eb,h_\eb,z_0}(\tau)=\Cal M^+_{\eb,T_\tau h_\eb,z_0}(0),
\end{equation}
see e.g. \cite{12} for the details.

\begin{remark}\label{Rem5.2} Arguing
in a standard way, it is not difficult to verify that expression
\eqref{5.18} together with the backward uniqueness proved in
Theorem \ref{Th2.3} allows indeed to endow the sets $\Cal
M^+_{\eb,h_\eb,z_0}(\tau)$ by the structure of a $C^1$-manifold
diffeomorphed to $V_+$, see \cite{4}. We however note that, in
contrast to local unstable manifolds, the global ones are not, in
general, {\it submanifolds} of $V^p_\eb(\omega)$ (even in the limit
autonomous case $\eb=0$). Indeed, if the limit equation
\eqref{4.1} possesses a homoclinic orbit to the equilibrium $z_0$,
then the corresponding global unstable manifold $\Cal M^+_{z_0}$
cannot be a submanifold of $V^p_0(\omega)$. Nevertheless, in
particular case where the limit parabolic equation possesses a
global Lyapunov function, the sets $\Cal M^+_{\eb,h_\eb,z_0}(\tau)$
are occurred to be {\it submanifolds} of the phase space $V^p_\eb$
for a sufficiently small $\eb$, see \cite{4} for the details.
\end{remark}


\section{The nonautonomous regular attractor}\label{s6}

In this section, using the theory of nonautonomous perturbations of
regular attractors (see \cite{12}), we obtain the detailed
description of the structure of attractors $\Cal A_\eb$, $\eb\ll1$,
of equations \eqref{3.16} in case where the limit parabolic
equation \eqref{4.1} is autonomous and possesses a global Lyapunov
function. In order to have the explicit expression for the Lyapunov
function, we assume that
\begin{equation}\label{6.1}
a=a^*\ \text{ and } \ f(u):=\nabla_u F(u),\
\text{ for some }\ F\in C^1(\R^k,\R).
\end{equation}
Then, system \eqref{4.1} possesses the following global Lyapunov
function:
\begin{equation}\label{6.2}
\Cal L(u_0):=\int_\omega a\Nx u_0.\Nx u_0+2F(u_0)+2\bar g.u_0\,dx.
\end{equation}
Let $\Cal R\subset V^p_0(\omega)$ be the set of equilibria of problem
\eqref{4.1}, i.e.
\begin{equation}\label{6.3}
\Cal R:=\{z\in V^p_0(\omega),\ \ a\Dx z-f(z)+\bar g=0\}.
\end{equation}
Our main assumption is that all of the equilibria of $\Cal R$ are
hyperbolic, i.e.
\begin{equation}\label{6.4}
\Cal R=\{z_i\}_{i=1}^N  \text{ and all of $z_i$ are hyperbolic,
see \eqref{5.2}}
\end{equation}
As known,
see \cite{4}, assumption \eqref{6.4} is satisfied for generic $\bar
g\in L^p(\omega)$ (belonging to some open and dense subset of
$L^p(\omega)$)). It is also well-known that, under above assumptions,
the global attractor $\Cal A_0$ of problem \eqref{4.1} possesses
the following description.
\begin{theorem}\label{Th6.1} Let the assumptions of Theorem \ref{Th4.1} hold
and let, in addition, \eqref{5.1}, \eqref{6.1} and \eqref{6.4}
 be satisfied. Then,
\par
1.\ Every solution $u(t)$, $t\in\R$, of \eqref{4.1}  belonging to
the attractor $\Cal A_0$ stabilizes as $t\to\pm\infty$ to different
equilibria $z_\pm\in\Cal R$:
\begin{equation}\label{6.5}
\lim_{t\to\pm\infty}\|u(t)-z_\pm\|_{V^p_0(\omega)}=0,
\end{equation}
where $z_+\ne z_-$.
\par
2. The attractor $\Cal A_0$ possesses the following description:
\begin{equation}\label{6.6}
\Cal A_0=\cup_{i=1}^N\Cal M^+_{z_i},
\end{equation}
where $\Cal M^+_{z_i}$ are finite-dimensional unstable manifold of
the equilibrium $z_i\in\Cal R$. Moreover, $\Cal
M^+_{z_i}$ is a $C^1$-submanifold of $V^p_0(\omega)$ which
is diffeomorphic to
$\R^{\kappa^+_i}$, where $\kappa_i^+$ is the instability
index of the equilibrium $z_i$.
\par
3.\ The set $\Cal A_0$ is an exponential attractor of the semigroup
  $S_t$ associated with equation \eqref{4.1}, i.e. there exist
a positive constant $\alpha$ and a monotonic function $Q$ such that,
for every bounded subset $B\subset V^p_0(\omega)$, the following
  estimate holds:
\begin{equation}\label{6.7}
\dist_{V^p_0(\omega)}\(S_tB,\Cal A_\eb\)\le
  Q(\|B\|_{V^p_0(\omega)})e^{-\alpha t}.
\end{equation}
\end{theorem}
The proof of this theorem can be found, e.g. in \cite{4}.
\par
The main task of this section is to construct the analogue
of the regular attractor $\Cal A_0$ for the 'nonautonomous' equation
\eqref{3.16} if $\eb>0$ is small enough. In this case, the
equilibria $z_i\in\Cal R$ should be replaced by the
almost-periodic solutions $u^\eb_{h_\eb,z_i}(t)$, $i=1,\cdots,N$,
constructed in Proposition \ref{Prop5.1} and, instead of the unstable manifolds
$\Cal M^+_{z_i}$, we should use the 'nonautonomous' unstable manifolds
$\Cal M^+_{\eb,h_\eb,z_i}(\tau)$, $\tau\in\R$, defined in \eqref{5.18}.
Analogously to \eqref{6.6}, for every $\tau\in\R$, $h\in\Cal H(g)$
and every sufficiently small $\eb>0$,
we define the attractor $\Cal A_{\eb,h_\eb}(\tau)$
 by the following expression:
\begin{equation}\label{6.8}
\Cal A_{\eb,h_\eb}(\tau):=\cup_{i=1}^N\Cal M^+_{\eb,h_\eb,z_i}(\tau),
\end{equation}
(we recall that, by definition, $h_\eb(t)=h(\eb^{-1}t)$).
 Then, due to \eqref{5.19},  the family of attractors
 $\Cal A_{\eb,h_\eb}(\tau)$, $\tau\in\R$,
is also strictly invariant with respect to the dynamical process
$U^\eb_{h_\eb}(t,\tau)$:
\begin{equation}\label{6.9}
\Cal A_{\eb,h_\eb}(t)=U^\eb_{h_\eb}(t,\tau)\Cal A_{\eb,h_\eb}(\tau).
\end{equation}
Moreover, the following theorem shows that this family is indeed a
nonautonomous regular attractor for the dynamical process
$U^\eb_{h_\eb}(t,\tau)$ if $\eb>0$ is small enough.
\begin{theorem}\label{Th6.2} Let the assumptions of Theorems \ref{Th4.1} and \ref{Th6.1}
hold. Then, there exists $\hat\eb_0'>0$,
 $0<\hat\eb_0'\le\hat\eb_0\ll1$
 such
that, for every $\eb\le\hat\eb_0'$ and $h\in\Cal H(g)$, the
following conditions are satisfied:
\par
1.\ Every bounded solution $u\in C_b(\R,V^p_\eb(\omega))$ of problem
\eqref{5.6} stabilizes as $t\to\pm\infty$ to different
almost-periodic 'equilibria'
of \eqref{5.6} constructed in Proposition \ref{Prop5.1}:
\begin{equation}\label{6.10}
\lim_{t\to\pm\infty}\|u(t)-u_{h_\eb,z_{i_\pm}}^\eb(t)\|_{V^p_\eb(\omega)}=0,\
\ i_\pm\in\{1,\cdots,N\}, \ i_-\ne i_+.
\end{equation}
In particular, $u(\tau)\in \Cal A_{\eb,h_\eb}(\tau)$, for every
$\tau\in\R$.
\par
2. For every fixed $\tau\in\R$,
the sets $\Cal M^+_{\eb,h_\eb, z_i}(\tau)$ are $C^1$-submanifolds of
$V^p_\eb(\omega)$ $C^1$-diffeomorphic to the unstable
manifolds $\Cal M^+_{z_i}$ of the limit autonomous parabolic
 problem \eqref{4.1} (which are independent of $\tau$ and $h$).
\par
3.  The sets $\Cal A_{\eb,h_\eb}(\tau)$, $\tau\in\R$, attract
    exponentially the images of all bounded subsets of $V^p_0(\eb)$,
    i.e, there exist a positive number $\alpha$ and a monotonic
    function $Q$ (which are independent of $\eb$, $h$, $t$ and $\tau$)
    such that
\begin{equation}\label{6.11}
\dist_{V^p_\eb(\omega)}\(U^\eb_{h_\eb}(t,\tau)B,\Cal A_{\eb,h_\eb}(t)\)\le
Q(\|B\|_{V^p_\eb(\omega)})e^{-\alpha(t-\tau)},
\end{equation}
for every bounded subset $B\subset V^p_\eb(\omega)$,  $h\in\Cal
H(g)$ and $t,\tau\in\R$, $t\ge\tau$.
\par
4. The attractors $\Cal A_{\eb,h_\eb}(\tau)$ tend to $\Cal A_0$ as
   $\eb\to0$ in the following sense:
\begin{equation}\label{6.12}
\sup\nolimits_{\tau\in\R}\sup\nolimits_{h\in\Cal H(g)}
\dist_{V^p_0(\omega)}^{sym}\(\Cal
   A_{\eb,h_\eb}(\tau),\Cal A_0\)\le \bar C
\big[\bar\alpha_{R_0}(\eb)\big]^{\kappa},
\end{equation}
where the constants $\bar C$, $R_0$ and $0<\kappa<1$ are independent
of $\eb$, $\bar\alpha_{R_0}(\eb)$ is the same as in Corollaries \ref{Cor4.1}
and \ref{Cor4.2} and
\begin{equation}\label{6.13}
\dist^{sym}_V(X,Y):=\max\{\dist_V(X,Y),\dist_V(Y,X)\}
\end{equation}
is the symmetric Hausdorff distance between sets $X$ and $Y$ in $V$.
\end{theorem}
{\it Sketch of proof.} The result of Theorem \ref{Th6.2} is a corollary of
the nonautonomous perturbation theory of regular attractors
developed in \cite{12,29}. In order to apply this theory to equation
\eqref{5.6}, we consider the discrete processes
\begin{equation}\label{6.14}
U^\eb_{h_\eb}(l,m): V^p_0(\omega)\to V^p_0(\omega)
\end{equation}
associated with this problem in phase space $V^p_0(\omega)$ which
is independent of $\eb$. Then, it follows Corollaries \ref{Cor2.4}, \ref{Cor4.1} and
\ref{Cor4.2} then these processes tend (together with their Frechet
derivative) as $\eb\to0$ to the semigroup $S_{l-m}$ associated with
the limit parabolic equation \eqref{4.1} in the phase space
$V^p_0(\omega)$ and this convergence is uniform with respect to
$h\in \Cal H(g)$. Moreover, estimate \eqref{2.29} guarantees the
uniform (with respect to $\eb$ and $h\in\Cal H(g)$) dissipativity
of these processes and Theorem \ref{Th2.3} gives the injectivity of all the
operators \eqref{6.14}. We also recall that the limit discrete
semigroup $S_n$, $n\in\Bbb N$, possesses the {\it regular}
attractor $\Cal A_0$ (due to Theorem \ref{Th6.1}). Then, arguing in a
standard way (see \cite{12}), we derive that discrete processes
\eqref{6.14} possess the nonautonomous regular attractors $\Cal
A_{\eb,h_\eb}(l)$, $l\in\Bbb Z$ (if $\eb>0$ is small enough) in
$V^p_0(\omega)$ which satisfy the discrete analogue of Theorem \ref{Th6.2}
(since all of the estimates formulated in Corollaries \ref{Cor2.4}, \ref{Cor4.1} and
\ref{Cor4.2} are uniform with respect to $h\in\Cal H(g)$ then estimates
\eqref{6.11} and \eqref{6.12} also hold for $\Cal
A_{\eb,h_\eb}(l)$ uniformly with respect to $h\in\Cal H(g)$).
\par
When the discrete nonautonomous attractors $\Cal A_{\eb,h_\eb}(l)$,
$l\in\Bbb Z$, for
processes \eqref{6.14} (which satisfy all of
the assertions of Theorem \ref{Th6.2}) are already constructed, we can extend
in a standard way
this result to the continuous case
by the following expression:
\begin{equation}\label{6.15}
\Cal A_{\eb,h_\eb}(\tau):=\Cal A_{\eb,T_\tau h_\eb}(0)
\end{equation}
which is an immediate corollary of \eqref{6.20}, see \cite{12} for
the details. Theorem \ref{Th6.2} is proved.

\qed
\begin{corollary}\label{Cor6.1} Let the assumptions of Theorem \ref{Th6.2}
hold. Assume also that the almost-periodic function $g(t)-\bar g$
(where $\bar g$ is a mean value  of $g(t)$) has a bounded primitive
\begin{equation}\label{6.16}
g(t)-\bar g=\Dt G(t),\ \ G\in L^2_b(\Omega).
\end{equation}
Then, estimate \eqref{6.12} can be improved as follows:
\begin{equation}\label{6.17}
\sup\nolimits_{\tau\in\R}\sup\nolimits_{h\in\Cal H(g)}
\dist_{V^p_0(\omega)}^{sym}\(\Cal
   A_{\eb,h}(\tau),\Cal A_0\)\le \bar C_1
\eb^{\kappa_1},
\end{equation}
where $\bar C_1$ is independent of $\eb$,
$\kappa_1:=\kappa_\delta\cdot\kappa/2$ and $\kappa_\delta$ is defined in
\eqref{4.12}.
\end{corollary}
Indeed, \eqref{6.17} is an immediate corollary of \eqref{6.12}
and Theorem \ref{Th4.3}.

\begin{corollary}\label{Cor6.2} Let the assumptions of
Theorem \ref{Th6.2} hold. Then the nonautono\-mous regular attractor $\Cal
A_{\eb,g}(t)$, $t\in\R$, of \eqref{5.6} and its uniform attractor
$\Cal A_\eb$ (constructed in Theorem \ref{Th3.1}) satisfy the following
relation:
\begin{equation}\label{6.18}
\Cal A_\eb=\cup_{h\in\Cal H(g)}\Cal
A_{\eb,h_\eb}(t)=\big[\cup_{t\in\R}\Cal A_{\eb,g_\eb}(t)\big]_{V^p_0(\omega)}
\end{equation}
and, consequently
\begin{equation}\label{6.19}
\dist^{sym}_{V^p_0(\omega)}\(\Cal A_\eb,\Cal A_0\)\le \bar
C\big[\alpha_{R_0}(\eb)\big]^\kappa.
\end{equation}
In particular, the uniform attractors $\Cal A_\eb$ tend to $\Cal A_0$
(upper and lower semicontinuous) as $\eb\to0$.
\end{corollary}
Indeed, the first equality in \eqref{6.18} is an immediate corollary
the first assertion of Theorem \ref{Th6.2} and description \eqref{3.7} of
uniform attractor $\Cal A_\eb$. The second inequality in \eqref{6.18}
can be easily verified using the exponential attraction property
\eqref{6.11} and the alternative definition of the uniform attractor
$\Cal A_\eb$ which is formulated in Remark \ref{Rem3.1}. Estimate \eqref{6.19}
follows immediately from \eqref{6.17} and \eqref{6.18}.
\begin{remark}\label{Rem6.1} The first assertion of Theorem \ref{Th6.2} can be
reformulated as follows:
problem \eqref{5.6} has exactly $N$ almost-periodic solutions
$u_{h_\eb,z_i}^\eb(t)$ which are localized near the equilibria $z_i\in\Cal
R$, $i=1,\cdots,N$, and every other bounded solution
$u\in C_b(\R,V^p_0(\omega))$ is a heteroclinic connection between two
different almost periodic solutions of this problem.
\end{remark}
 \begin{remark}\label{Rem6.2} We note that condition \eqref{6.16} is,
 obviously, always satisfied if the external force $g(t)$ is
{\it periodic} with respect to $t$. Thus, in case of periodic $g$,
we have estimate \eqref{6.17} for the symmetric distance
 between the perturbed ($\Cal A_{\eb,h}(t)$) and nonperturbed ($\Cal A_0$)
regular attractors
without any additional assumptions
 and (as a corollary) the following estimate is satisfied for
the uniform attractors:
\begin{equation}\label{6.20}
\dist^{sym}_{V^p_0(\omega)}\(\Cal A_\eb,\Cal A_0\)\le \bar
C_1\eb^{\kappa_1}.
\end{equation}
Unfortunately, in more general case of quasiperiodic or
almost-periodic external forces, condition \eqref{6.16} is not
satisfied automatically and should be verified, see e.g. \cite{9},
and \cite{18} for various sufficient conditions.
\end{remark}


\section{Appendix.  Uniform elliptic regularity in $L^p$-spaces}\label{a1}
In this Appendix, we consider the following singular perturbed elliptic
boundary value problem in a half-cylinder
$\Omega_+:=\R_+\times\omega$:
\begin{equation}\label{A.1}
a(\eb^2\Dt^2u+\Dx u)-\gamma\Dt u=h(t),\ \ u\DOM=0,\ \ u\tto=u_0,
\end{equation}
where $u=(u^1,\cdots,u^k)$ is a vector-valued function,
 $a$ and $\gamma$ are given constant matrices such that $a+a^*>0$ and
$\gamma=\gamma^*>0$ and the right-hand side $h$ belongs to
 $L^p(\Omega_+)$, $2\le p<\infty$.

The main result of this appendix is the following uniform (with
respect to $\eb$) maximal regularity estimate for the solutions of
 \eqref{A.1}.
\begin{theorem}\label{ThA.1} Let $u\in W^{(1,2),p}_{\eb}(\Omega_+)$
 be a solution of \eqref{A.1}. Then,
the following estimate holds:
\begin{equation}\label{A.2}
\|u\|_{W^{(1,2),p}_\eb(\Omega_+)}\le C\(\|u_0\|_{V^p_\eb(\omega)}+
\|h\|_{L^p(\Omega_+)}\),
\end{equation}
where the constant $C$ is independent of $\eb\in[0,\eb_0]$.
In particular, $V^p_\eb(\omega)$ is a uniform (with respect to $\eb$)
trace space for functions belonging to $W^{(1,2),p}_\eb(\Omega_+)$.
\end{theorem}
\begin{proof} The proof of estimate \eqref{A.2} is based on the
classical localization technique and on the multiplicators theorems
in Fourier spaces and is more or less standard (see e.g. \cite{17},
\cite{25}).
 That is the reason why, in order
to show that constant $C$ is indeed independent of $\eb$, we discuss below
only the principal points of this proof resting the details to the reader.
We start with the most simple case $\gamma=0$.
\begin{lemma}\label{LemA.1} Let $u$ be a solution of \eqref{A.1} with
$\gamma=0$. Then, the following estimate holds:
\begin{multline}\label{A.3}
C'\eb^{1/p}\(\|u_0\|_{W^{2-1/p,p}(\omega)}+
\eb\|\Dt u(0)\|_{W^{1-1/p,p}(\omega)}\)\le\\\le
\eb^2\|\Dt^2u\|_{L^p(\Omega_+)}+\|u\|_{L^p(\R_+,W^{2,p}(\omega))}\le
C\(\eb^{1/p}\|u_0\|_{W^{2-1/p,p}(\omega)}+\|h\|_{L^p(\Omega_+)}\),
\end{multline}
where the constants $C$ and $C'$ are independent of $\eb$.
\end{lemma}
Indeed, scaling the time $t=\eb t'$ and introducing the functions
 $\tilde u(t'):=u(t/\eb)$ and $\tilde h(t'):=h(t/\eb)$, we
 deduce
that the function $\tilde u$ satisfies equation \eqref{A.1} with
$\gamma=0$, $\eb=1$ and with the right-hand side $\tilde h$. Applying
the standard elliptic regularity theorem to this equation, see
 e.g. \cite{25}, we infer
\begin{multline}\label{A.4}
C'\(\|u_0\|_{W^{2-1/p,p}(\omega)}+
\|\partial_{t'}\tilde u(0)\|_{W^{1-1/p,p}(\omega)}\)\le\\\le
\|\tilde u\|_{L^p(\R_+,W^{2,p}(\omega))}+
\|\partial_{t'}^2\tilde u\|_{L^p(\Omega_+)}\le
C\(\|\tilde h\|_{L^p(\Omega_+)}+\|u_0\|_{W^{2-1/p,p}(\omega)}\).
\end{multline}
Returning to the time variable $t$, we derive estimate
\eqref{A.3}.
\par
In the next step, we consider the Hilbert case $p=2$.
\begin{lemma}\label{LemA.2} Let $p=2$ and
 $u\in W^{(1,2),2}_\eb(\Omega_+)$. Then, estimate \eqref{A.2}  holds.
\end{lemma}
\begin{proof} Indeed, multiplying equation \eqref{A.1} by $\eb^2\Dt^2
u+\Dx u$, integrating over $\Omega_+$, integrating by parts and using
that  $\gamma=\gamma^*$, we have
$$
\|a(\Dt^2u+\Dx u)\|^2_{L^2(\Omega_+)}+\frac{\eb^2}2(\gamma\Dt u(0),\Dt u(0))-\frac12(\gamma\Nx u(0),\Nx u(0))=\<h,\eb^2\Dt^2u+\Dx u\>_0
$$
and, therefore, since $a$ is non-degenerate and $\gamma>0$,
\begin{equation}\label{A.5}
\|\eb^2\Dt^2u+\Dx u\|_{L^2(\Omega_+)}^2+\eb^2\|\Dt u(0)\|_{L^2(\omega)}^2
\le C\(\|h\|_{L^2(\Omega_+)}^2+\|u_0\|_{W^{1,2}(\omega)}^2\),
\end{equation}
where $C$ is independent of $\eb$. Estimate \eqref{A.5}, together
with \eqref{A.3}, imply estimate \eqref{A.2} with $p=2$ and
 Lemma \ref{LemA.2} is proved.
\end{proof}
We are now ready to consider the general case $p>2$. We first note
that, due to the classical localization technique and estimate
\eqref{A.2} for $p=2$ (which is necessary in order to estimate the
subordinated terms appearing under the localization technique),
 it is sufficient to verify estimate
\eqref{A.2} only for equation
\begin{equation}\label{A.6}
a(\eb^2\Dt^2 u+\Dx u-u)-\gamma\Dt u=h,\ \ u\tto=u_0,\ \ u_0\DOM=0
\end{equation}
and only for two choices of the domain $\Omega$, namely, for 1)
$\omega=\R^n$ and 2)
$\omega_+=\R_+^{x_1}\times\R^{n-1}_{x_2,\cdots,x_n}$ (see e.g.
\cite{17} and \cite{25}). Moreover, we also note that the second
case of semi-space $\omega_+$ can be easily reduced to the first
one of the the whole space $\omega=\R^n$ by considering the odd
(with respect to $x_1$) solutions of \eqref{A.6} in $\omega=\R^n$.
Thus, there only remains to verify estimate \eqref{A.2} for
solutions of \eqref{A.6} in $\omega=\R^n$.
\par
In the next step, we reduce the problem of studying the elliptic system of
 equations \eqref{A.6} to
the analogous problem for the  scalar equation. In order to do so, it
 is convenient to extend the class of admissible solutions of
 \eqref{A.6} and consider also the {\it complex-valued} solutions
$u(t,x)=\Ree u(t,x)+i\Imm u(t,x)\in\Bbb C^k$, for every $(t,x)\in\Omega_+$.
Then, equation \eqref{A.6} is equivalent to the following one:
\begin{equation}\label{A.7}
\eb^2\Dt^2 u+\Dx u-u-\gamma'\Dt u=h,\ \ u\DOM=0,\ \ u\tto=u_0,
\end{equation}
where $\gamma':=a^{-1}\gamma$. Moreover, without loss of generality
we may assume that the matrix $\gamma'$ is reduced to its Jordan
normal form. Then, our conditions on matrices $a$ and $\gamma$
guaranties that the real parts of all eigenvalues of $\gamma'$ are
strictly positive:
\begin{equation}\label{A.8}
\sigma(\gamma')\subset\{\lambda\in\Bbb C,\ \ \Ree\lambda>0\}.
\end{equation}
Thus, \eqref{A.7} is a cascade system of scalar elliptic equations
coupled by the terms $\gamma'\Dt u$ and $\gamma'$ is in Jordan normal
form. That is why, it is sufficient to verify estimate \eqref{A.2}
only for scalar complex-valued elliptic equations of the form
\begin{equation}\label{A.9}
\eb^2\Dt^2 u+\Dx u-u-2(\alpha+i\beta)\Dt u=h,\ \ u\DOM=0,\ \ u\tto=u_0,
\end{equation}
where $\alpha,\beta\in\R$ and $\alpha>0$. We start with the case
$h=0$.
\begin{lemma}\label{LemA.3} Let $u_0\in V_\eb^p(\R^n)$ and let $u$ be a solution
 of \eqref{A.9} with $h=0$. Then, it satisfies uniform estimate
 \eqref{A.2}.
\end{lemma}
\begin{proof} Indeed, factorizing equation \eqref{A.9} (with
 $h\equiv0$), we obtain that the function $u(t)$ satisfies the
 following pseudodifferential equation:
\begin{equation}\label{A.10}
\Dt u=-A_\eb(1-\Dx)u,\ \ u\tto=u_0,
\end{equation}
where
\begin{equation}\label{A.11}
A_\eb(z):=-\frac{\alpha+i\beta-\sqrt{(\alpha+i\beta)^2+\eb^2z}}{\eb^2}\equiv
\frac z{\alpha+i\beta+\sqrt{(\alpha+i\beta)^2+\eb^2z}}
\end{equation}
and we take the branch of $\sqrt \cdot$ which is positive on
$\R_+$. Let us study equation \eqref{A.10}.
\begin{proposition}\label{PropA.1} The solution of \eqref{A.10} satisfies
\begin{equation}\label{A.12}
\|\Dt u\|_{L^p(\Omega_+)}+\|A_\eb(1-\Dx)u\|_{L^p(\Omega_+)}\le
 C\|u_0\|_{W^{2(1-1/p),p}(\R^n)},
\end{equation}
where $C$ is independent of $\eb$.
\end{proposition}
\begin{proof} We first consider the following nonhomogeneous
analogue of equation \eqref{A.10}:
\begin{equation}\label{A.13}
\Dt w+A_\eb(1-\Dx)w=h(t),\ \ w\tto=0,\  w\DOM=0,\ \ h\in L^p(\Omega_+)
\end{equation}
and verify that
\begin{equation}\label{A.14}
\|\Dt w\|_{L^p(\Omega_+)}\le C_3\|h\|_{L^p(\Omega_+)},
\end{equation}
where $C_3$ is independent of $\eb$. Indeed, let us
extend functions $w(t)$ and $h(t)$ by zero for $t<0$ and apply
the Fourier transform ($(t,x)\to \xi:=(\lambda,\xi')\in\R\times\R^n$)
to equation \eqref{A.13}. Then, we have
\begin{equation}\label{A.15}
\hat {(\Dt w)}(\xi)=K_\eb(\xi)\hat h(\xi),\ \
K_\eb(\xi):=\frac{i\lambda}{i\lambda+A_\eb(|\xi'|^2+1)}.
\end{equation}
According to the multiplicators theorem (see e.g. \cite{25}), in
order to verify estimate \eqref{A.14}, it is sufficient to prove
that
\begin{equation}\label{A.16}
\sup_{1\le i_1<\cdots<i_k\le n+1}\sup\nolimits_{\xi\in\R^{n+1}}
\big|\xi_{i_1}\cdots\xi_{i_k}\partial^k_{\xi_{i_1},\cdots,\xi_{i_k}}
K_\eb(\xi)\big|\le
C<\infty,
\end{equation}
where $C$ is independent of $\eb$. So, we need to verify
\eqref{A.16}. To this end, we note that, due to the assumption
$\alpha>0$, the following estimates hold:
\begin{multline}\label{A.17}
 \big|\Imm\sqrt{(\alpha+i\beta)^2+
\eb^2(|\xi'|^2+1)}\big|\le\kappa_1\sqrt{1+\eb^2(1+|\xi'|^2)}\le\\
\le\kappa_2\Ree\sqrt{(\alpha+i\beta)^2+\eb^2(|\xi'|^2+1)}\le
\kappa_3\sqrt{1+\eb^2(1+|\xi'|^2)}\\
\end{multline}
where $\kappa_i>0$, $i=1,2,3$, are independent of $\eb$
(indeed, these estimates can be easily verified by direct computations
based on the fact that $\alpha>0$). Estimates \eqref{A.17}, the
fact that $\alpha>0$ and definition \eqref{A.11}
immediately imply that
\begin{multline}\label{A.18}
\kappa_1'\big|\Imm A_\eb(|\xi'|^2+1)\big|\le
\frac{|\xi'|^2+1}{\sqrt{1+\eb^2(|\xi'|^2+1)}}\le\\\le
\kappa_2'\Ree A_\eb(|\xi'|^2+1)\le\kappa_3'
\frac{|\xi'|^2+1}{\sqrt{1+\eb^2(|\xi'|^2+1})}
\end{multline}
and, consequently
\begin{equation}\label{A.19}
 \kappa_1''\(|\lambda|+
|A_\eb(|\xi'|^2+1)|\)\le
\big|i\lambda+A_\eb(|\xi'|^2+1)\big|\le \kappa_2''\(|\lambda|+
|A_\eb(|\xi'|^2+1)|\),
\end{equation}
where the positive constants $\kappa_i'$ and $\kappa_i''$ are
independent of $\eb$. Moreover, due to \eqref{A.17} and \eqref{A.18}
\begin{multline}\label{A.20}
\big|\xi_{i_1}\cdots\xi_{i_k}
\partial_{\xi_{i_1}\cdots\xi_{i_k}}^kA_\eb(|\xi'|^2+1)\big|=\\
\frac{C_k(\eb^2|\xi_{i_1}|^2)\cdots (\eb^2|\xi_{i_{k-1}}|^2)}
{|(\alpha+i\beta)^2+\eb^2(|\xi'|^2+1)|^{k-1}}\cdot
\frac{|\xi_{i_k}|^2}{\sqrt{|(\alpha+i\beta)^2+\eb^2(|\xi'|^2+1)|}}\le
C_k'\big|A_\eb(|\xi'|^2+1)\big|
\end{multline}
holds, for every $2\le i_1<\cdots<i_k\le n+1$, where the constants
$C_k$ and $C_k'$ are independent of $\eb$.
There remains to note that
estimates \eqref{A.19} and \eqref{A.20} imply \eqref{A.16}.
Indeed, differentiating the kernel $K_\eb(\xi)$ with respect to $\xi_{i_1},\cdots,\xi_{i_k}$ and using \eqref{A.20}, we see that, for $2\le i_1<\cdots< i_k\le n+1$,
$$
|\xi_{i_1}\cdots\xi_{i_k}\partial_{\xi_{i_1},\cdots\xi_{i_k}}^k K_\eb(\xi)|\le C_k\(\sum_{l=1}^k\frac{|\lambda|\cdot|A_\eb(|\xi'|^2+1)|^l}{(|\lambda|+|A_\eb(|\xi'|^2+1))^{l+1}}\)\le C
$$
and the analogous uniform (with respect to $\eb\to0$) estimate for the case where $i_1=1$ also holds.
 Thus, estimate \eqref{A.16} holds and, therefore, \eqref{A.14} is also verified.
\par
Let us now prove  estimate \eqref{A.12}. To this end, we
fix an extension $v(t)$ of the initial data $u_0$ inside of $\Omega_+$
in such way that
\begin{equation}\label{A.21}
\|\Dt v\|_{L^p(\Omega_+)}+
\|v\|_{L^p(\R_+,W^{2,p}(\omega))}\le C_1\|u_0\|_{W^{2(1-1/p),p}(\R^n)},
\end{equation}
where $C_1$ is independent of $u_0$ (such an extension exists due
to the classical trace theorems, see \cite{25}) and introduce a
function $w(t):=u(t)-v(t)$ which, obviously, satisfies equation
\eqref{A.13} with $h(t):=\Dt v(t)+A_\eb(1-\Dx)v(t)$. Thus, thanks
to \eqref{A.14} and \eqref{A.21}, it is sufficient to verify that
\begin{equation}\label{A.22}
\|A_\eb(1-\Dx)v(t)\|_{L^p(\R^n)}\le C_2\|v(t)\|_{W^{2,p}(\R^n)},
\end{equation}
where $C_2$ is independent of $\eb$ and $t$. But this estimate can be easily
verified using the multiplicators theorem and estimates \eqref{A.18} and \eqref{A.20} (in the same way as it was done in the
proof of estimate \eqref{A.14}). Proposition \ref{PropA.1} is proved.
\end{proof}
We are now able to finish the proof of Lemma \ref{LemA.3}. Indeed, according to
Proposition \ref{PropA.1}, every solution $u(t)$ of \eqref{A.9} with $h=0$
satisfies estimate \eqref{A.12}. Interpreting now the term
$2(\alpha+i\beta)\Dt u$ in equation \eqref{A.9} as the right-hand
side and using Lemma \ref{LemA.1}, we derive that $u(t)$  satisfies indeed
estimate \eqref{A.2} with $h=0$ which finishes the proof of Lemma \ref{LemA.3}.
\end{proof}
In particular, Lemma \ref{LemA.3} implies that $V^p_\eb(\R^n)$ is a uniform
trace space for functions from $W^{(1,2),p}_\eb(\Omega_+)$ at $t=0$.
Indeed, the solving operator $T_+: u_0\to u$ for \eqref{A.9}
with $h=0$ can be considered as uniform (with respect to $\eb$)
extension operator for functions from $V^p_\eb(\R^n)$ to
 $W^{(1,2),p}_\eb(\Omega_+)$ and the inverse estimate
\begin{equation}\label{A.23}
\|u(0)\|_{V^p_\eb(\R^n)}\le C\|u\|_{W^{(1,2),p}_\eb(\Omega_+)}
\end{equation}
is an immediate of Lemma \ref{LemA.1} and the standard trace theorem for
the 'parabolic' space $W^{(1,2),p}_0(\Omega_+)$.
\par
We are now ready to finish the proof of Theorem \ref{ThA.1}. As it was shown
before, in order to do so, it is sufficient to verify estimate
\eqref{A.2} for equation \eqref{A.9} in $\Omega_+:=\R_+\times\R^n$.
Moreover, due to Lemma \ref{LemA.3} and due to the fact that $V^p_\eb$ is
a uniform trace space for functions from $W^{(1,2),p}_\eb$, it is
sufficient to verify that every solution $u\in W^{(1,2),p}_\eb(\R^{n+1})$  of
\begin{equation}\label{A.24}
\eb^2\Dt u+\Dx u-u-2(\alpha+i\beta)\Dt u=h(t), \ \ t\in\R, \ \ x\in\R^n
\end{equation}
satisfies the estimate
\begin{equation}\label{A.25}
\|u\|_{W^{(1,2),p}_\eb(\R\times\R^n)}\le C\|h\|_{L^p(\R^{n+1})},
\end{equation}
where $C$ is independent of $\eb$. Applying
 the Fourier transform to \eqref{A.24}, we infer
\begin{equation}\label{A.26}
\hat u(\lambda,\xi')=
\(\eb^2\lambda^2+|\xi'|^2+1-(\alpha+i\beta)i\lambda\)^{-1}\hat
 h(\lambda,\xi')
\end{equation}
Applying the multiplicators theorem to \eqref{A.26} (as we did in the
proof of Proposition \ref{PropA.1}), we derive estimate \eqref{A.25} which
finishes the proof of Theorem \ref{ThA.1}.
\end{proof}
To conclude, we formulate several standard corollaries of the proved
theorem the rigorous proof of which is left to the reader.
\begin{corollary}\label{CorA.1} Let $h\in L^p_b(\Omega_+)$ and
let $u\in W^{(1,2),p}_{\eb, b}(\Omega_+)$ be a solution of \eqref{A.1}. Then,
the following estimate holds for every $T\ge0$:
\begin{equation}\label{A.27}
\|u\|_{W^{(1,2),p}_\eb(\Omega_T)}^p\le
C\|u_0\|_{V^p_\eb(\omega)}^pe^{-\alpha T}+
C\int_0^\infty e^{-\alpha|T-t|}\|h(t)\|_{L^p(\omega)}^p\,dt,
\end{equation}
where positive constants $C$ and $\alpha$ are independent of $\eb$,
$u_0$, $T$
and $u$.
\end{corollary}
Indeed, multiplying equation \eqref{A.1} by
$\phi_{T,\alpha}(t):=1/\cosh(\alpha(T-t))$, where $\alpha>0$ is a
sufficiently small number, and applying Theorem \ref{ThA.1} to the function
$w_{T,\alpha}(t):=\phi_{T,\alpha}(t)u(t)$, we obtain \eqref{A.27}
after the standard estimations.
\par
The next corollary gives the standard interior (with respect to $t$)
estimate for solutions
of \eqref{A.1}.
\begin{corollary}\label{CorA.2} Let $h\in L^p_b(\Omega_+)$ and
let $u\in W^{(1,2),p}_{\eb, b}(\Omega_+)$ be a solution of \eqref{A.1}. Then,
the following estimate holds for every $T\ge0$:
\begin{multline}\label{A.28}
\|u\|_{W^{(1,2),p}_\eb(\Omega_T)}\le\\\le
C(\|h\|_{L^p(\Omega_{T-1/2,T+3/2})}+\|u\|_{L^2(\Omega_{T-1/2,T+3/2})}+
\chi(1-2T)\|u_0\|_{V^p_\eb(\omega)}),
\end{multline}
where $\Omega_{T_1,T_2}:=[\max\{T_1,0\},T_2]\times\omega$, $\chi(z)$ is the
Heaviside function and the constant $C$ is independent of $\eb$, $T$
and $u$.
\end{corollary}
Indeed, the prove of \eqref{A.28} is based on multiplication of
equation \eqref{A.1} by the special cut-off function $\psi_T(t)$
which vanishes for $t\notin[T-1/2,T+3/2]$ and equals one for
$t\in[T,T+1]$ and on application of Theorem \ref{ThA.1} to the function
$u_T(t):=\psi_T(t)u(t)$ and can be derived  in a standard way.


\end{document}